\documentclass[a4paper,10pt]{amsart}

\usepackage{amsmath}     % ams math
\usepackage{amsthm}      % ams environments
\usepackage{amssymb}     % ams symbols
\usepackage{adjustbox}

\usepackage{graphicx}    % including graphics
\usepackage{colonequals} % colonequals etc.

\usepackage{hyperref}    % hyperreferences
\usepackage{cleveref}    % clever references
\usepackage[square,numbers,sort]{natbib}

\usepackage{todonotes}   % todo notes
\presetkeys{todonotes}{fancyline, color=gray!30}{}

\newcommand{\abs}[1]{{\left\lvert#1\right\rvert}}        % absolute value signs
\newcommand{\norm}[1]{\left\lVert#1\right\rVert}         % norm
\newcommand{\cbrackets}[1]{\left\{#1\right\}} % { }
\newcommand{\scp}[1]{\left\langle#1\right\rangle}        % scalarproduct

\DeclareMathOperator*{\deter}{det}            % determinant
\DeclareMathOperator*{\argmin}{arg\,min}      % argmin

\newcommand{\vcs}[1]{{\boldsymbol#1}}        % vector symbol font
\newcommand{\vc}[1]{{\mathbf#1}}             % \boldsymbol #1 \mathbf #1 use:\vc{q}

\newcommand{\Id}{\text{Id}}                   % identity
\newcommand{\vn}{\vc n}                       % normal field
\newcommand{\vvarphi}{\vcs\varphi}            % deformation
\newcommand{\vvarphih}{\vvarphi_h}            % discrete deformation
\newcommand{\avarphi}{z}                % algebraic deformation
\newcommand{\varphik}{\avarphi^k}            % deformation iterate
\newcommand{\varphikone}{\avarphi^{k+1}}     % next deformation iterate
\newcommand{\vP}{\vc P}                       % piola stress
\newcommand{\vt}{\vc t}                       % Cauchy stress
\newcommand{\fext}{\vc f}                     % external volume forces
\newcommand{\ttext}{\vc t}                    % external traction forces
\newcommand{\Fext}{\mathcal F}                % external volume potential
\newcommand{\Text}{\mathcal G}                % external traction potential
\newcommand{\W}{\mathcal W}                % hyperelastic energy
\newcommand{\J}{\mathcal{J}}                  % energy
\newcommand{\aJ}{J}                  % algebraic energy
\newcommand{\aJk}{\aJ^k}                  % algebraic energy of iterate k
\newcommand{\aJi}{\aJ^i}                  % algebraic energy of iterate k
\newcommand{\aJkone}{\aJ^{k+1}}                  % algebraic energy of iterate k
\newcommand{\vartetk}{\vartheta^k}
\newcommand{\varteti}{\vartheta^i}
\newcommand{\vartetkone}{\vartheta^{k+1}}
\newcommand{\avNormal}{\vc n_p}           % nodally averaged normal
\newcommand{\gap}{g_{h}}                      % discretised gap function
\newcommand{\vu}{ u}                          % correction in filter scheme
\newcommand{\vuk}{\vu^k}                      % correction solution of exact qp in filter scheme
\newcommand{\tvu}{\bar\vu}                    % correction in transformed coordinates
\newcommand{\tvuk}{ \bar\vu^k}           % correction in transformed coordinates
\newcommand{\mk}{m^{k}}                       % local approximation of the Lagrangian/energy
\newcommand{\imk}{\widetilde{m}^k}          % local inexact approximation of the Lagrangian/energy
\newcommand{\mkT}{\mk_{T}}                    % local approximation in transformed coordinates
\newcommand{\fkT}{\fk_{\T}}                   % exactly transformed rhs
\newcommand{\HkT}{\Hk_{\T}}                   % exactly transformed Hessian
\newcommand{\iHk}{\widetilde{H}^k}                   % inexactly transformed Hessian
\newcommand{\iHkT}{\widetilde{H}_T^k}                   % inexactly transformed Hessian
\newcommand{\T}{T}                            % exact mortar transformation
\newcommand{\iT}{\widetilde T}                % inexact mortar transformation
\newcommand{\mone}{m_{1}}                       % non-mortar contact dofs
\newcommand{\mtwo}{m_{2}}                       % mortar contact dofs

\newcommand{\ag}{c}                           % the algebraic non-penetration constraint
\newcommand{\nmMat}{D}                        % exact non-mortar matrix
\newcommand{\nmMatN}{\nmMat_N}                 % exact non-mortar matrix in normal coordinates
\newcommand{\nmMatT}{\nmMat_T}                 % exact non-mortar matrix in tangential coordinates
\newcommand{\inmMat}{\widetilde\nmMat}        % inexact non-mortar matrix
\newcommand{\mMat}{M}                         % exact non-mortar matrix
\newcommand{\vnmu}{\vu^1_C}                    % non-mortar displacement dofs
\newcommand{\tvnmu}{\bar{\vu}^1_C}       % non-mortar displacement in transformred coordinates
\newcommand{\tvnmuO}{\bar{\vu}^1_{C,0}}  % transf. nm. disp only first component
\newcommand{\vmu}{\vu^2_C}                     % mortar displacement dofs
\newcommand{\avarphinm}{\avarphi^1_C}         % non-mortar deformation
\newcommand{\avarphim}{\avarphi^2_C}          % mortar deformation
\newcommand{\fk}{f^k}                        % the gradient of energy at iterate k
\newcommand{\Hk}{H^k}                        % Hessian of energy at iterate k
\newcommand{\Tk}{\T^k}                       % exact transformation at k
\newcommand{\Deltak}{\Delta^k}

\newcommand{\vuknu}{\vu^{\nu}}                 % old tnnmmg iterate
\newcommand{\vuknuhalf}{\vu^{\nu+\frac 12}}    % pre-smoothed iterate
\newcommand{\vw}{w}                           % GS intermediat iterate
\newcommand{\Qnu}{Q^{\nu}}    % truncation matrix of smoothed iterate
\newcommand{\corr}{v}                         % tnnmmg coarse correction
\newcommand{\tcorr}{\bar{\corr}}                         % tnnmmg coarse correction
\newcommand{\R}{\mathbb R}                   % real numbers
\newcommand{\Rdn}{\R^{dn}}                   % real numbers to the power dn
\newcommand{\Na}{\mathbb N}                  % natural numbers
\newcommand{\GammaNm}{\Gamma_C^1}             % non-mortar boundary
\newcommand{\GammaM}{\Gamma_C^2}              % mortar boundary
\newcommand{\gammaNm}{\gamma_C^1}             % deformed non-mortar boundary
\newcommand{\gammaH}{\gamma_{h}}              % discrete deformed non-mortar boundary
\newcommand{\gammaM}{\gamma_C^2}              % deformed mortar boundary
\newcommand{\Matplus}{\text{Mat}^+}           % matrix with positive determinant
\newcommand{\K}{\mathcal K}                   % feasible set
\newcommand{\Th}{\mathcal{T}_h}               % triangulation with mesh size
\newcommand{\N}{\mathcal N}                   % set of vertices of a discretised set
\newcommand{\F}{\mathcal F}                   % filter
\newcommand{\ConeH}{M_h^+}                    % discrete mortar space

\newcommand{\Hhalf}{H^{\frac 12}}                    % Trace space
\newcommand{\HoneDomega}{\vc H^1_D(\Omega)}          % H^1_D Omega

\newcommand{\A}{\mathcal{A}}                         % Active set
\newcommand{\kscd}{\kappa_{\textnormal{\text{scd}}}} % sufficient cauchy decrease
\newcommand{\kteta}{\kappa_{\vartheta}}              % constant for infeasibility iteration
\newcommand{\kT}{\kappa_{T}}                         % bound for inverse transformation

\newtheorem{assumption}{Assumption}
\crefname{assumption}{Assumption}{Assumptions}

\theoremstyle{plain}
\newtheorem{theorem}{Theorem}[section]

\newtheorem{remark}[theorem]{Remark}
\newtheorem{lemma}[theorem]{Lemma}
\crefname{lemma}{Lemma}{Lemmata}

\theoremstyle{definition}
\newtheorem{definition}[theorem]{Definition}

\begin{document}

\author[J. Youett]{Jonathan Youett}
\address[Jonathan Youett]{Department of Mathematics and Computer Science\\
  Freie Universit\"at Berlin\\
Arnimallee 6, 14195 Berlin}
\thanks{This work has been done within the DFG \textsc{Matheon} project CH1 funded by \textsc{ECMath}}
\email{youett@math.fu-berlin.de}

\author[O. Sander]{Oliver Sander}
\address[Oliver Sander]{Institute of Numerical Mathematics\\
  Technische Universit\"at Dresden\\
  Zellescher Weg 12--14, 01069 Dresden}
\email{oliver.sander@tu-dresden.de}

\author[R. Kornhuber]{Ralf Kornhuber}
\address[Ralf Kornhuber]{Department of Mathematics and Computer Science\\
  Freie Universit\"at Berlin\\
Arnimallee 6, 14195 Berlin}
\email{kornhuber@math.fu-berlin.de}

\title[Filter--Trust-Region Method for contact problems]{A globally convergent filter--trust-region method for large deformation contact problems}

\begin{abstract}
We present a globally convergent method for the solution of frictionless large deformation contact problems for hyperelastic materials.
The discretisation uses the mortar method which is known to be more stable than node-to-segment approaches.
The resulting non-convex constrained minimisation problems are solved using a filter--trust-region scheme,
and we prove global convergence towards first-order optimal points.
The constrained Newton problems are solved robustly and efficiently using a Truncated Non-smooth Newton Multigrid (TNNMG) method with a Monotone Multigrid (MMG) linear correction step.
For this we introduce a cheap basis transformation that decouples the contact constraints.
Numerical experiments confirm the stability and efficiency of our approach.

\end{abstract}

\maketitle

\section{Introduction}

Although large deformation contact problems arise in many important applications,
only very few methods today can solve them fast and robustly.
All of these methods have their advantages and disadvantages.

Discretisation of such problems leads to constrained non-convex minimisation problems.
The prevailing methods for these problems are primal--dual active set strategies~\cite{Hartmann_Ramm:2008,Hesch_Betsch:2009,Popp:2011} and penalty methods~\cite{Puso_Laursen:2004}.
For both methods only local convergence can be expected~\cite{Wohlmuth:2005}.
Furthermore, the resulting linearised Newton problems can be indefinite due to the non-convexity of the strain energy.

In this work we construct a filter--trust-region method~\cite{Fletcher_Leyffer:2002} for the
constrained non-convex minimisation problem.
The filter technique ensures asymptotic fulfilment of the non-linear non-penetration constraints by rejecting iterates that are neither improving the energy nor the infeasibility compared to all previous iterates.
The trust-region method provides a natural way to handle indefiniteness of the linearised problems.
We show that the modifications we need to make to the method to apply it to contact problems stay within the realm of the general
filter--trust-region convergence theory, and hence we obtain global convergence of the method
to first-order stationary points.

A priori, the Newton problems of a filter--trust-region method are quadratic minimisation problems with convex inequality constraints.
Such problems are generally expensive to solve.
We extend an efficient multigrid strategy originally introduced for contact problems in small strain elasticity~\cite{Wohlmuth_Krause:2003}
to the case of large strains.
This requires rewriting the inequality constraints as sets of bound constraints.
The inequality constraints consist of two parts: the trust-region constraint and the linearised contact condition.
We define the trust-region in terms of the max-norm.
With this choice, the trust-region constraints form a set of bound constraints by construction.
To decouple the contact constraints we extend the technique used in \cite{Wohlmuth_Krause:2003} to the finite strain case.
The idea there is to construct a basis transformation that replaces the nodal basis at
the contact boundaries by a system of relative movements.  The construction of this transformation
requires the solution of a linear system involving a contact surface mass matrix
(the non-mortar matrix) during each Newton-type iteration.
The additional computational cost of this is negligible because the size of the mass matrix
is much smaller than the overall problem size, and grows with a lower order.
The transformation also leads to a slight modification of the Newton matrix, but we show that this modification does not influence the convergence behaviour of the overall method.

In previous work, the Truncated Non-smooth Newton Multigrid (TNNMG) method has been shown to be
very fast and effective for quadratic minimisation problems with bound constraints such as
small-strain contact problems and obstacle problems~\cite{Graeser_Kornhuber:2009,Graeser_Sack_Sander:2009}.
Since we have found a way to uncouple the contact constraints for finite-strain contact problems,
we can also harness the performance of TNNMG for the Newton problems of a finite-strain contact problem.
Unfortunately, this only works if the quadratic models are convex.
For the non-convex case, we extend the TNNMG method by combining it with a Monotone Multigrid
(MMG) method for the linear correction step.
The MMG method will handle the trust-region constraints (which have a comparatively simple structure)
while the more complicated contact constraints will be left to the TNNMG step.
The resulting scheme is globally convergent even for indefinite trust-region problems.
At the same time, we observe multigrid-type convergence rates in numerical experiments.

This paper is organised as follows: In \Cref{sec:continuous_contact} the static large deformation contact problem
is described and its weak formulation is derived. In \Cref{sec:discretisation} we summarise the mortar discretisation
of the problem, which we use because it avoids most instabilities and unphysical oscillations of the node-to-segment approaches~\cite{Puso_Laursen:2004}.
As a stepping stone, we then construct a locally convergent, efficient solver based on sequential quadratic programming in \Cref{sec:solution}. We introduce the TNNMG multigrid algorithm and the constraints
decoupling strategy needed to solve the quadratic constrained Newton problems.
To globalise the local SQP solver, in \Cref{sec:filter} we then describe the filter--trust-region algorithm and the combined TNNMG/MMG scheme for the solution of the linearised problems.
We show global convergence of both methods.
The final \Cref{sec:numerics} is dedicated to a numerical example.

\section{Static large deformation contact problems}
\label{sec:continuous_contact}
In this section we will briefly summarise the equations of equilibrium of two non-linear hyperelastic bodies subject to mutual contact. A more detailed introduction can be found, e.g., in~\cite{Laursen:2003}.

\subsection{Strong formulation}
\label{ssec:strong_formulation}
Let $\Omega^i\subset\R^d$, $i=1,2,\,d=2,3$ denote the disjoint reference configurations of two deformable objects.
Assume that the boundaries of the $\Omega_i$ are such that the outer unit normal fields $\vn^i_R:\partial\Omega^i\to\R^d$ exist everywhere.
Let the boundaries be decomposed into disjoint relatively open sets  $\partial\Omega^i = \overline\Gamma^i_D\cup\overline\Gamma^i_N\cup\overline\Gamma_C^i$ corresponding to Dirichlet, Neumann, and contact boundaries.
We assume that $\Gamma_D^i$ has positive $(d-1)$-dimensional measure for $i=1,2$, and that $\Gamma_D^1$ is compactly embedded in $\partial\Omega^1\setminus\overline{\GammaNm}$.
%The latter assumption ensures that the Dirichlet condition and the non-penetration constraint are decoupled in which case the trace space on the contact boundary is given by $\Hhalf(\GammaNm)$ directly rather than $\Hhalfzero(\GammaNm)$.

In the following, unindexed variables are used to denote quantities defined over both objects.
For example $\Omega = \Omega^1\cup\Omega^2$ denotes the reference configuration of both bodies together.
Neglecting the inertia terms, the balance of linear momentum yields the following system of partial differential equations in reference coordinates for the deformation function $\vvarphi \colonequals (\vvarphi^1,\,\vvarphi^2):\Omega\to\R^d$
\begin{alignat}{2}
  \label{eq:pde_elasticity}
  \nonumber\operatorname{div}\vP(\vvarphi) + \fext &= 0 &&\text{in }\Omega,\\
  \vP(\vvarphi) \vn_R &= \ttext \quad&&\text{on }\Gamma_N, \\
  \nonumber\vvarphi &= \vvarphi_D \quad&&\text{on }\Gamma_D.
\end{alignat}
Here, $\vP:\Omega\to \Matplus(d)$ is the first Piola--Kirchhoff stress field, and $\Matplus(d)$ is the set of $d\times d$ matrices with positive determinant.
The functions $\fext\in \vc L^2(\Omega)$
and $\ttext \in \vc L^2(\Gamma_N)$ are prescribed external volume and traction force densities, which are assumed to be independent of the deformation.
The function $\vvarphi_D\in C(\Gamma_D)^d$ specifies the Dirichlet boundary conditions.
We will only consider hyperelastic continua, i.e., materials for which there exists a stored energy functional $\W:\Omega\times\Matplus(d)\to\R$, $(x,F) \mapsto \W(x,F)$, that links the stresses to the deformation via
\begin{equation}
  \label{eq:constitutive_law}
  \frac{\partial \W}{\partial F}(\cdot,\nabla\vvarphi) = \vP(\vvarphi).
\end{equation}
We assume that the hyperelastic energy is penalising any violation of the orientation-preserving condition
\begin{equation}
  \label{eq:orientation_preserving}
 \deter\nabla\vvarphi(x)>0 \quad\forall\, x\in\Omega,
\end{equation}
in the sense that
\begin{equation*}
  \W(x,\nabla\vvarphi) \to \infty\quad\text{ if } \deter\nabla\vvarphi(x)\searrow 0.
\end{equation*}
As a consequence, we will not explicitly enforce \eqref{eq:orientation_preserving} as a hard constraint.

The subsets $\Gamma_C^i$ denote the parts of the boundaries where contact may occur.
Contact constraints are naturally formulated on the deformed domain.
For $i=1,2,$ let $\vn^i$ denote the outer unit normal field on the deformed contact boundary $\gamma_C^i \colonequals \vvarphi^i(\Gamma_C^i)$.
Modelling of non-penetration can be done in several ways, depending on which projection is chosen to identify the contact surfaces with each other.
Earlier papers used the closest-point projection from $\gamma_C^1$ to $\gamma_C^2$~\cite{Laursen:2003,Laursen_Simo:1993,Wriggers:2006}.
Recently, using the projection along $\vn^1$ has become more popular~\cite{Hesch_Betsch:2009,Puso_Laursen:2004,Popp:2011,Wriggers:2009}.
In the following we only consider the closest-point projection approach, but others can be used equally well.
The deformed contact boundaries are identified with each other through the projection $\Phi:\gammaNm\to\gammaM$
\begin{equation*}
  \label{eq:normal_projection}
  \Phi(s) \colonequals \argmin_{r\in\gammaM}\norm{s-r}.
\end{equation*}
The resulting distance function or \textit{signed gap function} $g : \gammaNm\to\R$ is given by
\begin{equation}
  \label{eq:gap_function}
  g(s) \colonequals \vn^{2}(\Phi(s))\cdot(s - \Phi(s)),
\end{equation}
where we have used the fact that $s-\Phi(s)$ is orthogonal to $\gammaM$ at $\Phi(s)$.
With these definitions, non-penetration of the bodies is enforced by requiring
\begin{equation}
  \label{eq:strong_constraint}
  g(s)\geq 0 \quad \forall\, s\in\gammaNm,
\end{equation}
cf.\ \Cref{fig:nonpenetration}.
\begin{figure}
\centering
  \def\svgwidth{\textwidth}
  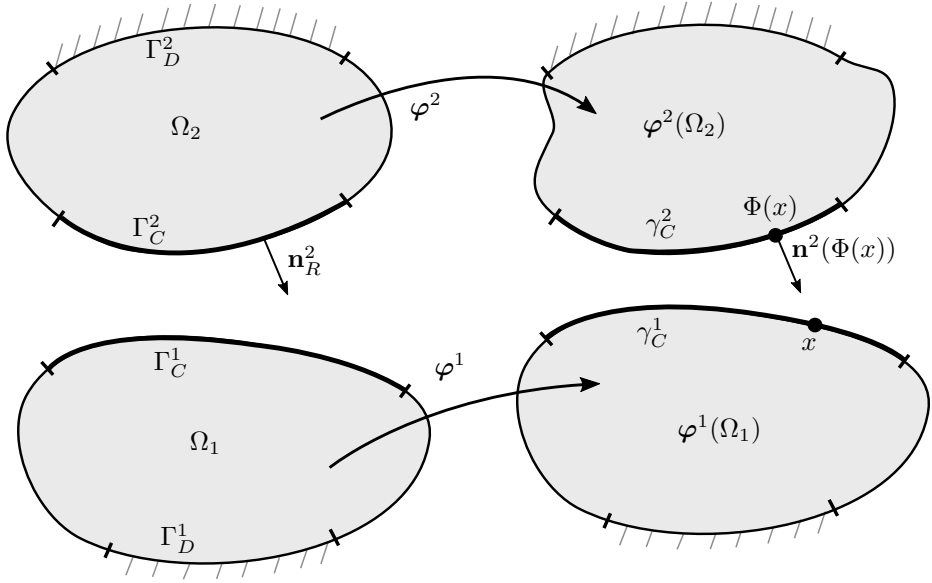
   \caption{Reference and deformed configuration of the two bodies}
\label{fig:nonpenetration}
\end{figure}

So far the non-penetration constraint was derived only from a kinematical point of view.
To investigate the effect of these constraints on the elastic system we examine the resulting contact forces.
Consider the \textit{Cauchy stress tensor} $\vcs\sigma(\vvarphi)\colonequals \deter(\nabla\vvarphi)^{-1}\vP(\vvarphi)\nabla\vvarphi^T$, which expresses the stress relative to the deformed configuration $\vvarphi(\Omega)$.
The \textit{Cauchy boundary traction}
\begin{equation*}
  \vt_C\colonequals\vcs\sigma(\vvarphi^1) \vn^1
  \label{eq:Piola_traction}
\end{equation*}
then represents the contact forces on $\gamma_C^1$.
It can be decomposed into normal and tangential parts $t_N$ and $\vt_T$ with respect to $\vn^1$.
We consider frictionless contact only so the tangential traction $\vt_T$ at the contact boundary vanishes.
The contact normal stresses fulfil the \emph{Karush--Kuhn--Tucker (KKT) conditions}
\begin{equation*}
  \label{eq:kkt_contact}
  t_N \leq 0,\quad g\geq 0,\quad g\cdot t_N = 0\quad \text{on } \gammaNm,
\end{equation*}
where the first one states that traction is a pressure, the second one is \eqref{eq:strong_constraint}, and the last one is the \emph{complementary condition}~\cite{Laursen:2003}.
%
%-----------------------------------------------------------------------------------
%
\subsection{Weak formulation}
\label{ssec:weak_formulation}
The equilibrium configurations of hyperelastic continua are characterised as stationary points of the energy functional
\begin{equation*}
  \J(\vvarphi)\colonequals\int_{\Omega}\W(x,\nabla\vvarphi) - \Fext(\vvarphi)\,dx - \int_{\Gamma_N} \Text(\vvarphi)\,ds,
  %\label{eq:strain_energy_functional}
\end{equation*}
where $\W$ is the hyperelastic energy density \eqref{eq:constitutive_law}, and $\Fext$ and $\Text$ are potentials of the external forces.
Stable configurations are the local minimisers of this energy~\cite[Theorem 4.1-2]{Ciarlet:1988}.
%Then the weak formulation of unconstrained static hyperelasticity is given by:
%
%\begin{equation}
%  \label{eq:weak_elasticity}
%\vvarphi\in\HoneDomega:\quad \J(\vvarphi) \leq \J(\vv) \quad\forall\,\vv\in\HoneDomega,
%\end{equation}
%
Existence of minimisers has been shown for the case of a poly-convex and coercive strain energies~\cite[Theorem 7.7-1]{Ciarlet:1988}.
The corresponding first-order optimality condition is the weak form of the elasticity problem~\eqref{eq:pde_elasticity}

We now add the contact constraints.
In a Sobolev space setting, the non-penetration constraint \eqref{eq:strong_constraint} takes the form
\begin{equation}
  g(s)\geq 0\quad \text{for almost all }\,s\in\gammaNm,
  \label{eq:sobolev_constraint}
\end{equation}
and similarly for the other two KKT conditions.
In anticipation of the mortar discretisation we rewrite this condition in a variationally consistent form.
Let $\HoneDomega$ denote the Sobolev space of $d$-valued weakly differentiable functions fulfilling the Dirichlet boundary conditions in the sense of traces.
We assume that the gap function $g$ is smooth enough such that
\begin{equation*}
  \vvarphi\mapsto g(\vvarphi)
\end{equation*}
maps every $\HoneDomega$ function to a function in $W\colonequals \Hhalf(\gammaNm)$.
We denote the dual trace space by
\begin{equation*}
  M \colonequals  \Hhalf(\gammaNm)',
\end{equation*}
and the cones of positive functions and dual functionals by
\begin{align*}
  W^+&\colonequals\cbrackets{v\in W : v\geq 0\,\text{ a.e.}},\\
  M^+ &\colonequals\cbrackets{\mu\in M : \scp{\mu,v}\geq 0,\,\forall\,v\in W^+},
\end{align*}
where $\scp{\cdot,\cdot}$ denotes the dual paring of $M$ and $W$.
Now, the resulting weak formulation of the non-penetration constraint \eqref{eq:sobolev_constraint} is given by
\begin{equation*}
  \scp{\mu,g(\vvarphi)} \geq 0\quad\forall\,\mu\in M^+.
  %\label{eq:weak_constraint}
\end{equation*}
The equivalence of this to \eqref{eq:sobolev_constraint} is shown in~\cite{Wohlmuth:2011}.
%\begin{proposition}
%  The weak non-penetration constraint \eqref{eq:weak_constraint} is equivalent to \eqref{eq:sobolev_constraint}
%\end{proposition}
%\begin{proof} \cite[Proposition 1.3.4]{Youett:2016}\end{proof}
%
We denote by
\begin{equation*}
  \K \colonequals \cbrackets{\vvarphi\in\HoneDomega :\,\scp{\mu,g(\vvarphi)} \geq 0\quad\forall\,\mu\in M^+}
\end{equation*}
the closed non-convex set of feasible deformations.
The weak formulation of the large deformation contact problem now reads:
\begin{equation}
  \label{eq:weak_contact}
  \text{Find a local minimiser } \vvarphi \text{ of }\J \text{ in } \K.
\end{equation}
To our knowledge the question of existence of solutions is still open.
\section{Discretisation}
\label{sec:discretisation}
In this section we will describe the discretisation of the minimisation problem \eqref{eq:weak_contact} using first-order Lagrangian finite elements, and mortar elements for the contact constraints.
Let $\Th$ be a shape-regular grid of the bodies $\Omega \colonequals \Omega^1\cup\Omega^2$, and $\N(\Th)$ the set of vertices.
%Further denote by $\gammaH^i$ the restriction of $\Th$ onto the deformed contact boundaries $\gamma_C^i$, $i=1,2$.
The space of $d$-valued first-order finite elements is $\vc S_h = (S_h)^d$, and for each node $p\in\N(\Th)$ the scalar nodal basis function corresponding to $p$ is denoted by $\psi_p\in S_h$.
We discretise the hyperelasticity problem \eqref{eq:weak_contact} by replacing the solution space $\HoneDomega$ by the finite dimensional subspace $\vc S_{D,h}\colonequals\vc S_h\cap\HoneDomega$.

\subsection{Dual mortar discretisation of the contact constraints}
\label{ssec:mortar_discretisation}
We use dual mortar functions~\cite{Wohlmuth:2001} to discretise the mortar cone $M^+$,
but Lagrange functions can be used equally well. For a given discrete deformation
$\vvarphi_h \in \vc S_{D,h}$, let $\gamma_h^i$ be the grid of the deformed contact boundary
obtained by restricting $\Th$ to the reference contact boundary $\Gamma_C^i$, and then
deforming this restriction using $\vvarphi_h$.
We denote the basis of the Lagrange multiplier space by
\begin{equation}
  \Theta_h^\varphi \colonequals  \big\{\theta_p^\varphi: p\in \N(\gammaH^1) \big\}.
  \label{eq:discrete_cone_span}
\end{equation}
The discrete mortar cone $\ConeH \not\subset M^+$ is then given by
\begin{equation*}
  \ConeH \colonequals \Bigl\{\mu_h\in\operatorname{span}\Theta_h^\varphi :\, \int_{\gammaH^1}\mu_h(s)\,v_h(s)\,ds\geq 0 \quad \forall\,v_h\in S_h(\gammaH^1),\,v_h\geq 0\Bigr\}.
  \label{eq:discretised_cone}
\end{equation*}
This leads to the weak non-penetration constraint
\begin{equation*}
  \int_{\gammaH^1}g(s)\,\mu_h(s)\,ds \geq 0\quad \forall\,\mu_h\in \ConeH,
\end{equation*}
which, considering the definition \eqref{eq:gap_function} of the gap function $g$, is
\begin{equation}
  \label{eq:weak_inserted_constraint}
  \int_{\gammaH^1}\vn^2(\Phi(s))\cdot(s - \Phi(s))\,\mu_h(s)\,ds \geq 0\quad \forall\,\mu_h\in \ConeH.
\end{equation}
As the normal field of a piecewise polynomial surface, $\vn^2$ is not continuous on $\gammaM$.
We therefore replace it by a smoothed normal field $\vn_h$.
Define vertex normals by averaging the adjacent face normals, i.e., for each vertex $p\in\N(\gammaM)$ with neighbouring faces $\mathcal{E}(p)$ on the contact boundary we set
\begin{equation*}
  \avNormal
  \colonequals
  \frac{\sum_{e\in \mathcal{E}(p)} \vn_e}{\big\lVert \sum_{e\in \mathcal{E}(p)} \vn_e\big\rVert},
  %\label{eq:nodally_averaged_normal}
\end{equation*}
where $\vn_e$ is the face normal of $e$ at the corner $p$.
The discretised normal field $\vn_h$ is then defined as the finite element function
\begin{equation}
  \vn_h \colonequals \sum_{p\in\N(\gammaH^2)} \psi_p\avNormal,
  \label{eq:discretised_normal_field}
\end{equation}
and we replace $\vn^2$ in \eqref{eq:weak_inserted_constraint} with $\vn_h$.
This continuous approximation yields a smoother behaviour when sliding occurs compared to using discontinuous element normals, cf.~\cite{Puso_Laursen:2004}.
The resulting discrete non-penetration constraint with
\begin{equation}
  \label{eq:strong_disc_gap}
  \gap(s) \colonequals\vn_h(\Phi(s))\cdot(s - \Phi(s))
\end{equation}
reads
\begin{equation}
  \label{eq:discrete_gap}
  \int_{\gammaH^1} \gap(s)\,\mu_h(s)\,ds\geq 0,\quad \mu_h\in\ConeH.
\end{equation}
We denote the corresponding discrete feasible set by
\begin{equation*}
  \K_h \colonequals \Bigl\{\vvarphih\in\vc S_{D,h} :\, \int_{\gammaH^1}\gap(s)\,\mu_h(s)\,ds\geq 0,\quad\forall\mu_h\in\ConeH\Bigr\}.
  \label{eq:discrete_inequality_set}
\end{equation*}
Summarising, the discrete problem is given by:
\begin{equation}
  \text{Find a local minimiser } \vvarphih \text{ of }\J \text{ in } \K_h.
  %\vvarphih\in\K_h:\quad \J(\vvarphih)\leq \J(\vv_h)\quad\forall\,\vv_h\in\K_h.
  \label{eq:discrete_problem}
\end{equation}
As for the non-discrete case~\eqref{eq:weak_contact}, the existence of solutions of \eqref{eq:discrete_problem} appears to be an open question.
\subsection{Algebraic contact problem}
\label{ssec:algebraic_problem}
For the rest of this paper we will denote the $p$-th component of a (block-)vector $v$ by $v_{p}$, the $p$-th row of a matrix $A$ by $A_p$, and the $(p,q)$-th entry of a (block-)matrix $A$ by $A_{pq}$.
The algebraic representation of the finite-strain contact problem is derived using the canonical isomorphism $I:\Rdn\to\vc S_h$ (where $n\colonequals\abs{\N(\Th)}$) that identifies finite element functions with their coefficient (block-)vectors.
The algebraic energy is then given by
\begin{equation*}
  \aJ : \R^{dn} \to \R,
  \qquad
  \aJ(\avarphi)\colonequals \int_{\Omega} \W(I(\avarphi))\,dx -b^T\avarphi,
  %\label{eq:alg_energy}
\end{equation*}
with $b\in\Rdn$ given component-wise by
\begin{equation*}
  (b_{p})_i \colonequals \int_\Omega \fext\,e_i\psi_p\,dx  + \int_{\Gamma_N}\ttext\,e_i\psi_p\,ds \quad 1\leq p \leq n,\,0\leq i < d,
\end{equation*}
where $e_i$ denotes the $i$-th Euclidean basis vector.
The non-penetration constraint \eqref{eq:discrete_gap} is represented algebraically by a function $\ag:\Rdn\to\R^{\mone}$, with $\mone\colonequals\abs{\N(\gammaH^1)}$, defined by testing the weak constraint \eqref{eq:discrete_gap} with the mortar basis functions \eqref{eq:discrete_cone_span}
\begin{equation}
  \ag_{q}(\avarphi)\colonequals \int_{\gammaH^1}\gap(s)\theta_q(s)\,ds\quad 1\leq q\leq \mone.
  \label{eq:algebraic_gap}
\end{equation}
Details on how to assemble the algebraic constraints can be found in~\cite{Popp:2011}.
Summarising, the non-convex algebraic contact problem reads:
\begin{equation}
  \text{Find a local minimiser $\avarphi$ of $\aJ$ in $K$},
  %\avarphi\in \text{K}:\quad \aJ(\avarphi)\leq \aJ(v)\quad\forall\,v\in \text{K},
 \label{eq:algebraic_problem}
\end{equation}
where
\begin{equation*}
  K \colonequals \cbrackets{\avarphi\in\Rdn :\, c_q(\avarphi)\geq 0 \quad 1\leq q\leq \mone}.
\end{equation*}
\section{Inexact SQP multigrid methods for contact problems}
\label{sec:solution}
In this section we show how \eqref{eq:algebraic_problem} can be solved locally using \textit{sequential quadratic programming} (SQP).
We propose a basis transformation that decouples the linearised constraints for each quadratic sub-problem.
The transformed problems can then be solved robustly and efficiently using a Truncated Non-smooth Newton Multigrid (TNNMG) method.
The transformation involves minor modifications to the tangent stiffness matrices that do not harm the overall convergence properties.

\subsection{Sequential Quadratic Programming}
\label{ssec:sqp}
Consider the constrained optimisation problem \eqref{eq:algebraic_problem}.
The first-order optimality conditions are given by the following theorem.
\begin{theorem}[\cite{Nocedal:2006}, Theorem\,12.1]
    \label{thm:kkt}
    Let $\avarphi^*$ be a local minimiser of \eqref{eq:algebraic_problem}.
    If the rows of the active constraint Jacobian, i.e., those rows $p$ of $\nabla\ag(\avarphi^*)$ for which  $\ag_{p}(\avarphi^*)=0$, are linearly independent,
    then there exists a Lagrange multiplier $\lambda\in\R^{\mone}$ such that
    \begin{equation}
      \label{eq:discrete_kkt}
      \begin{aligned}
        \nabla\aJ(\avarphi^*) + \lambda^T\nabla \ag(\avarphi^*) &= 0,\\
        \ag(\avarphi^*)&\geq 0,
      \end{aligned}
    \end{equation}
  and
    \begin{equation*}
      \lambda_p \leq 0,\quad \lambda_{p} \ag_{p}(\avarphi^*) = 0, \quad 1\leq p\leq \mone.
    \end{equation*}
\end{theorem}
The SQP method is derived by applying Newton's method to the first-order optimality system \eqref{eq:discrete_kkt} and eliminating the Lagrange multiplier.
In the following let upper indices $k\in\Na$ be the iteration number of the Newton method, and introduce a quadratic model energy by
\begin{equation}
  \label{eq:model_energy}
  \mk(\vu)\colonequals \nabla\aJ(\varphik)^T\vu + \frac 12 \vu^T \Hk\vu,
\end{equation}
with $\Hk\in\R^{dn\times dn}$ symmetric.
The SQP constraints are derived by replacing the constraint $c(z) \ge 0$ with its linearisation
at $z^k$
\begin{equation}
  \label{eq:lin_constraints}
    \nabla\ag(\varphik)\vu + \ag(\varphik)\geq 0,
\end{equation}
where $u \in \R^{dn}$ is the argument of $m^k$.
Then, the Newton problems can be reformulated as quadratic minimisation problems for the correction $\vuk\in\Rdn$
\begin{equation}
    \tag{QP}
    \min_{\vu\in\Rdn} \mk(\vu),\quad \nabla\ag(\varphik)\vu + \ag(\varphik)\geq 0.
  \label{eq:sqp_subproblems}
\end{equation}
Local linear convergence of this scheme can be proven if $\Hk$ is a symmetric positive definite approximation of the Hessian of the Lagrangian
\begin{equation}
  \Hk \approx \nabla^2\aJ(\varphik) + (\lambda^k)^T\nabla^2 \ag(\varphik),
  \label{eq:hessian_lagrangian}
\end{equation}
see \cite[Theorem\,18.7]{Nocedal:2006}.
\subsection{Multigrid methods for bound-constrained quadratic minimisation problems}
\label{ssec:mg_obstacle_problems}
In an SQP method, solving the constraint quadratic problems \eqref{eq:sqp_subproblems} is by far the most costly part.
We will solve these problems with multigrid efficiency using the Truncated Non-smooth Newton Multigrid (TNNMG) method~\cite{Graeser_Sack_Sander:2009}.
To illustrate the method we assume for the rest of this section that the local models $\mk$ are strictly convex.

Consider the quadratic functional \eqref{eq:model_energy} with $\Hk\in\R^{dn\times dn}$ symmetric and positive definite.
For simplicity we drop the superscript $k$ for this section.
Also, for this section only, we assume that the linear constraints decouple into bound constraints.
Hence, we want to find the unique minimiser of $m$ subject to
\begin{equation}
  \label{eq:bound_constraints}
  a_i\leq \vu_i\leq b_i\quad 1\leq i\leq dn,
\end{equation}
where the $a_i$ may be $-\infty$ and the $b_i$ may be $+\infty$.
One iteration step of TNNMG for this problem can be separated into the following four sub-steps:
Let $\vu^\nu\in\Rdn$ be a given iterate.

\medskip

\noindent\textit{1.\ Projected Gauss-Seidel step}\\
Set $\vw_0 = \vu^\nu$; then for $p = 1,\,\ldots,\,dn$, set
\begin{equation}
  \begin{aligned}
    &\alpha_p = \argmin_{a_{p}\leq \alpha + \vu^\nu_{p}\leq b_{p}}\, m(\vw_{p-1} + \alpha e_p),\\
  &\vw_p = \vw_{p-1} + \alpha_p e_p,
\end{aligned}
  \label{eq:gs_problem_convex}
\end{equation}
where $e_p$ is the $p$-th Euclidean basis vector.\\
Denote by $\vu^{\nu+\frac 12} \colonequals \vw_{dn}$ the resulting pre-smoothed iterate.

\medskip

\noindent\textit{2.\ Truncated linear correction}\\
To accelerate the convergence, the smoothing is followed by a linear correction step for the defect problem
\begin{equation*}
  \min_{\corr\in\Rdn}\frac 12 \corr^TH\corr - r^T\corr,
\end{equation*}
where the residual is given by
\begin{equation*}
  r \colonequals \nabla\aJ(\avarphi)  - H\vu^{\nu + \frac 12}.
\end{equation*}
For this step the active components
\begin{equation*}
  \A(\vu) \colonequals \Bigl\{p\in\cbrackets{1,\ldots,dn} : \vu_{p} = a_{p} \text{ or } \vu_{p} = b_{p}\Bigr\}
\end{equation*}
are truncated~\cite{Graeser_Kornhuber:2009}, i.e., temporarily frozen.
This is achieved by multiplying the defect problem with the truncation matrix
\begin{equation}
  \label{eq:truncation_matrix}
  \Qnu\in\R^{dn\times dn}:\quad\Qnu_{pq} \colonequals \begin{cases} 1 & p=q \text{ and }p\notin\A(\vu^{\nu+\frac 12}),\\ 0 & \text{else.}\end{cases}
\end{equation}
The linear truncated defect problem therefore reads
\begin{equation}
  \label{eq:convex_defect_problem}
  v^\nu \colonequals \argmin_{\corr\in\Rdn} \frac 12 \corr^T Q^\nu H Q^\nu \corr - (r^\nu Q^\nu)^T \corr.
\end{equation}
Note that the defect problem \eqref{eq:convex_defect_problem} is unconstrained.
For the approximate solution of this problem on the space spanned by the inactive
components one (or a few) geometric or algebraic linear multigrid step(s) is used.

\medskip

\noindent\textit{3.\ Projection}\\
The resulting correction $\corr^\nu$ may violate the defect constraints.
To ensure feasibility it is projected back onto the defect obstacles in the $l^2$-sense, i.e., we define $\hat\corr^\nu$ by
\begin{equation*}
  \hat\corr^\nu_i \colonequals \begin{cases} b_i -\vu^{\nu+\frac12}_i & \text{if } \corr^\nu_i > b_i -\vu^{\nu+\frac12}_i,\\
   a_i -\vu^{\nu+\frac12}_i & \text{if } \corr^\nu_i < a_i -\vu^{\nu+\frac12}_i,\\
   \corr_i^\nu & \text{else}.
 \end{cases}
\end{equation*}

\medskip

\noindent\textit{4.\ Line search}\\
The projection in Step~3 can lead to an increase of model energy.
To ensure monotonicity of the algorithm a line search is performed
\begin{equation}
  \alpha^\nu=\argmin_{\alpha\in\R}\,\mk(\vuknuhalf+\alpha\hat\corr^\nu),\quad\text{s.t. } \vuknuhalf + \alpha\hat\corr^\nu \text{ admissible.}
\label{eq:line_search_problem}
\end{equation}
This one-dimensional constrained quadratic problem can be solved analytically.
As a result we obtain $\vu^{\nu+1}\colonequals \vu^{\nu+\frac 12} + \alpha^\nu\hat\corr^\nu$ with
\begin{equation*}
  m(\vu^{\nu+1})\leq m(\vu^{\nu +\frac 12})\leq m(\vu^\nu).
\end{equation*}
Global convergence of the algorithm follows immediately from the convergence of the pre-smoothing Gauss--Seidel step and the monotonicity.
\begin{theorem}
  [\cite{Graeser_Kornhuber:2009}, Theorem.\,6.4]
  Suppose that $H\in\R^{dn\times dn}$ is symmetric positive definite, and the constraints
  have the form~\eqref{eq:bound_constraints}. Then
  the TNNMG method converges globally to a minimiser of \eqref{eq:model_energy} subject to \eqref{eq:bound_constraints}.
\end{theorem}
\subsection{Decoupling the constraints}
\label{ssec:decoupling_exact_constraints}
In this section we will construct a basis transformation of $\Rdn$ that decouples the linearised contact constraints~\eqref{eq:lin_constraints}.
This generalises an idea from~\cite{Wohlmuth_Krause:2003}, which did the same
in the infinitesimal strain framework.
We start by considering $\nabla\ag(\avarphi)$ in more detail:
The linearisation
\begin{equation*}
  \delta \ag_p(\avarphi)\vu \colonequals  \lim_{t\to 0} \ag_p(\avarphi + t\vu) \quad \vu\in\Rdn
\end{equation*}
of the $p$-th component of the algebraic contact constraint \eqref{eq:algebraic_gap} in the direction of $\vu\in\Rdn$ can be divided into three parts
\begin{equation}
  \begin{aligned}
      \delta \ag_p(\avarphi) = &\int_{\gammaH^1}\delta\vn_h(\Phi(s))\cdot(s - \Phi(s))\,\theta_p(s)\,ds\\
      + &\int_{\gammaH^1}\vn_h(\Phi(s))\cdot\delta\Bigl[(s - \Phi(s))\,\theta_p(s)\Bigr]\,ds\\
  + &\int_{\gammaH^1}\vn_h(\Phi(s))\cdot(s - \Phi(s))\,\theta_p(s)\,\delta ds.
\end{aligned}
  \label{eq:linearised_constraints}
\end{equation}
The first part is the linearisation of the nodally averaged normal field \eqref{eq:discretised_normal_field}. In the continuous case this term vanishes due to the colinearity of the normal $\vn^2(\Phi(s))$ with the closest point projection $s -\Phi(s)$, see \cite{Laursen:2003}.
The second part is the linearisation of the discretised gap function \eqref{eq:strong_disc_gap} and the mortar basis function,
and the third summand is the linearisation of the deformation dependent integral domain, which we denote by $\delta ds$.

Let $\mtwo\colonequals\abs{\N(\GammaM)}$ be the number of vertices on the contact boundary $\GammaM$, and as before $\mone\colonequals\abs{\N(\GammaNm)}$.
In the following we assume for simplicity that the coefficient vectors $\vu\in\Rdn$ are ordered such that $\vu = (\vnmu,\,\vmu,\,\vu^I)$, where $\vnmu\in\R^{d\mone}$ and $\vmu\in\R^{d\mtwo}$
are the degrees of freedom on the contact boundaries $\GammaNm$ and $\GammaM$ respectively, and $\vu^I$ denotes all other degrees of freedom.
Then, the algebraic form \eqref{eq:linearised_constraints} of the constraint Jacobian $\nabla\ag(\avarphi)$ can be split into a \emph{non-mortar} and a \emph{mortar} part, corresponding to the linearisations with respect to $\avarphinm$ and $\avarphim$, respectively
\begin{equation*}
  \nabla\ag(\avarphi) = \begin{pmatrix} \nmMat(\avarphi) & \mMat(\avarphi) & 0 \end{pmatrix},
\end{equation*}
where $\nmMat(\avarphi) \colonequals \frac{\partial \ag(\avarphi)}{\partial\avarphinm} \in \R^{\mone\times d\mone},\quad \mMat(\avarphi)\colonequals\frac{\partial\ag(\avarphi)}{\partial\avarphim}\in\R^{\mone\times d\mtwo}$ are sparse block-matrices given by
\begin{equation*}
  \begin{aligned}
 % \nmMat(\avarphi) \colonequals \frac{\partial \ag(\avarphi)}{\partial\avarphinm} \in \R^{\mone\times d\mone},\quad \mMat(\avarphi)\colonequals\frac{\partial\ag(\avarphi)}{\partial\avarphim}\in\R^{\mone\times d\mtwo},
    \nmMat(\avarphi)_{pq} = &\int_{\gammaH^1}\vn_h(\Phi(s))\frac{\partial}{\partial\avarphi^1_q}\Bigl[(s - \Phi(s))\,\theta_p(s)\Bigr]\,ds\\
    + &\int_{\gammaH^1}\vn_h(\Phi(s))\cdot(s - \Phi(s))\,\theta_p(s)\,\frac{\partial}{\partial\avarphi^1_q} ds\in \R^{1\times d},\\
    \mMat(\avarphi)_{pq} =
      &\int_{\gammaH^1}\frac{\partial}{\partial\avarphi^2_q}
      \bigl[\vn_h(\Phi(s))\bigr](s - \Phi(s))\,\theta_p(s) - \vn_h(\Phi(s))\frac{\partial}{\partial\avarphi^2_q}
      \Bigl[\Phi(s)\Bigr]\theta_p(s)\,ds\\
  + &\int_{\gammaH^1}\vn_h(\Phi(s))\cdot(s - \Phi(s))\,\theta_p(s)\,\frac{\partial}{\partial\avarphi^2_q} ds\in\R^{1\times d},
  \end{aligned}
\end{equation*}
and $0$ denotes a $\mone\times d(n-\mone-\mtwo)$ zero matrix.
The algebraic linearised constraints \eqref{eq:lin_constraints} then take the form
\begin{equation}
  \label{eq:lin_algebraic_constraints}
 \nmMat(\avarphi)\vnmu + \mMat(\avarphi)\vmu \geq -\ag(\avarphi).
\end{equation}

In our aim to decouple these constraints we first separate the normal from the tangential components.
Let $O(\avarphi)\in\R^{d\mone\times d\mone}$ be the block-diagonal matrix consisting of Householder transformations $O_{11},\ldots,O_{\mone\mone}$ such that $O_{pp}(\avarphi)\in\R^{d\times d}$ rotates the first Euclidean basis vector $e_1\in\R^d$ onto the normal $\vn_h$ at the projected vertex $\Phi(p)\in\gamma_h^2$, for all $p\in\N(\gammaH^1)$.
We use $O(z)$ to transform the non-mortar matrix by
\begin{equation}
  \bigl(\nmMat(\avarphi) O(\avarphi)\bigr)_{pq} \equalscolon  \bigl(\overbrace{\nmMatN(\avarphi)_{pq}}^{\in\R} \quad \overbrace{\nmMatT(\avarphi)_{pq}}^{\in\R^{d-1}} \bigr).
  \label{eq:normal_nonmortar}
\end{equation}
In the normal part $\nmMatN(\avarphi)\in\R^{\mone\times \mone}$, the first component of each $(1\times d)$-block of $D(\avarphi)O(\avarphi)$ is collected.
Analogously, the $d-1$ tangential components are collected in $\nmMatT(\avarphi)\in\R^{\mone\times (d-1)\mone}$.
The crucial insight of \cite{Wohlmuth_Krause:2003} was to see that the contact constraints can be decoupled by inverting $\nmMatN$.
For small-strain contact problems this could be trivially achieved, because the
biorthogonality of the dual mortar basis lead to a diagonal matrix $\nmMatN$.
In the finite-strain setting, $\nmMatN$ is sparse but no longer diagonal.
We suppose that the matrix remains invertible, for all relevant configurations $z$.
For the sake of the argument we use its inverse $\nmMatN^{-1}$ now and comment later on how to compute it efficiently.

Consider the following deformation-dependent transformation $T(\avarphi)\in\R^{dn\times dn}$
\begin{equation}
    \T(\avarphi) \colonequals  \begin{pmatrix}  O(\avarphi)K(\avarphi)  & -O(\avarphi)L(\avarphi) & 0\\ 0 & \Id  & 0 \\ 0& 0 & \Id
  \end{pmatrix},
  \label{eq:exact_mortar_transformation}
\end{equation}
where the $(d\times d)$-block-matrices $K(\avarphi),L(\avarphi)$ are given component-wise by
\begin{equation*}
    K_{pq}(\avarphi) \colonequals
    \begin{pmatrix}
      -(\nmMatN^{-1})_{pq} &
      -(\nmMatN^{-1}\nmMatT)_{pq} \\[0.5em]
      0      & \delta_p^q\Id^{(d-1)\times (d-1)}
   \end{pmatrix},
   \quad L_{pq}(\avarphi) \colonequals
   \begin{pmatrix} (\nmMatN^{-1}\mMat)_{pq} \\[0.5em] 0^{(d-1)\times d} \end{pmatrix}.
\end{equation*}
The inverse of $\T$ is sparse and has the form
\begin{equation*}
  \T(\avarphi)^{-1} = \begin{pmatrix}  U(\avarphi)O(\avarphi)  & -V(\avarphi) & 0\\ 0 & \Id  & 0 \\ 0& 0 & \Id
  \end{pmatrix},
\end{equation*}
where
\begin{equation*}
    U_{pq} \colonequals
    \begin{pmatrix}
      -(\nmMatN)_{pq} &
      -(\nmMatT)_{pq} \\[0.5em]
      0      & \delta_p^q\,\Id^{(d-1)\times (d-1)}
    \end{pmatrix}
    \quad
    \text{and}
    \quad
    V_{pq} \colonequals
    \begin{pmatrix}
    \mMat_{pq} \\[0.5em]
    0^{(d-1)\times d} \end{pmatrix}.
\end{equation*}
\begin{lemma}
    \label{lem:transformed_basis}
    In the transformed coordinates
\begin{equation*}
  \tvu = \T^{-1}(\avarphi)\vu,
\end{equation*}
the linearised contact constraints \eqref{eq:lin_algebraic_constraints} take the form
\begin{equation}
  \tvnmuO \leq \ag(\avarphi),
  \label{eq:transformed_constraints}
\end{equation}
where
%$\tvnmuO\in\R^m$ is given component wise by
%consists of the first component of each transformed contact degree of freedom, i.e.,
$\tvnmuO$ is the vector that contains the first of each block of $d$ degrees of freedom on $\gammaH^1$.
%\begin{equation*}
%  \N_0(\gammaH^1)\colonequals\Bigl\{p\in\cbrackets{1,\ldots,dn} :\,p \text{ first component of a node }q\in\N(\gammaH^1)\Bigr\}.
%\end{equation*}
\end{lemma}
\begin{proof}
  We omit the dependencies on $\avarphi$ for simplicity.
  The linearised constraints \eqref{eq:lin_algebraic_constraints} transform according to
  \begin{align*}
     (\nabla\ag)\,\T &=
    \begin{pmatrix} \nmMat & \mMat & 0\end{pmatrix}\, \T \\
    &=\bigl(DOK \quad \bigl[\mMat-DOL\bigr] \quad 0 \bigr).
  \end{align*}
  The first column of this is an $(\mone\times d\mone)$-matrix.
  It can be simplified by noting that for any $p,q=1,\ldots,\mone$
  \begin{align*}
    (DOK)_{pq} &= \sum_{j=1}^m \begin{pmatrix}(\nmMatN)_{pj} & (\nmMatT)_{pj} \end{pmatrix}
    \begin{pmatrix}
      -(\nmMatN^{-1})_{jq} & -(\nmMatN^{-1}\nmMatT)_{jq} \\[0.5em]
         0   & \delta_j^q\Id\\
   \end{pmatrix}\\
   &=\begin{pmatrix} -\delta_p^q &  \bigl[-(\nmMatT)_{pq} + (\nmMatT)_{pq}\bigr] \end{pmatrix} = \begin{pmatrix} -\delta_p^q & 0^{1\times (d-1)} \end{pmatrix}.
  \end{align*}
The second column vanishes since for $p\in\cbrackets{1,\ldots,\mone}$,\,$q\in\cbrackets{1,\ldots,\mtwo}$ we have
  \begin{equation*}
    (DOL)_{pq} = \sum_{j=1}^{\mone} \begin{pmatrix}(\nmMatN)_{pj} & (\nmMatT)_{pj} \end{pmatrix}
    \begin{pmatrix} (\nmMatN^{-1}\mMat)_{jq} \\ 0 \end{pmatrix} = \mMat_{pq},
  \end{equation*}
  where we have used the relationship \eqref{eq:normal_nonmortar}. Therefore, if $\vu$ is a vector such that \eqref{eq:lin_constraints} holds, we obtain that \eqref{eq:transformed_constraints}, and vice versa.
\end{proof}
In transformed coordinates, sub-problem \eqref{eq:sqp_subproblems} turns into
\begin{equation}
    \tag{TQP}
    \min_{\tvu\in\Rdn}\, \mkT(\tvu),\quad (\tvnmuO)_p \leq \ag_p(\varphik), \quad p=1,\ldots,\mone,
  \label{eq:sqp_transformed_subproblem}
\end{equation}
with transformed quadratic energy
\begin{equation*}
  \begin{aligned}
    &\mkT(\tvu)\colonequals \fkT\tvu + \frac 12 \tvu^T \HkT\tvu,\\
 &\fkT\colonequals \nabla\aJ(\varphik)^T\T(\varphik),\quad \HkT\colonequals \T(\varphik)^T\Hk \T(\varphik).
 \end{aligned}
\end{equation*}
If the transformed model energy is strictly convex, this quadratic minimisation problem with bound constraints can be solved by the TNNMG method of the previous section.
To avoid having to assemble \eqref{eq:exact_mortar_transformation} on all grid levels,
only the pre-smoothing Gauss--Seidel step is applied to the decoupled formulation \eqref{eq:sqp_transformed_subproblem}.
The truncated defect problem and coarse grid correction are computed in Euclidean coordinates.
\subsection{Avoiding the inverse non-mortar matrix}
\label{ssec:inexact_constraints}

The decoupling strategy of the previous section uses the explicit inverse of the sparse matrix $\nmMatN(\avarphi)$, whose size corresponds to the number of degrees of freedom on the non-mortar contact boundary $\gammaH^1$.
While the inverse itself can be computed in reasonable time using a direct sparse solver, it leads to
a considerable increase of density of the tangent stiffness matrix $\HkT$ compared to the untransformed matrices $\Hk$.
This severely slows down the multigrid solver.
In the following we show how the matrix inversion and the resulting density increase can be avoided while conserving the convergence of the SQP method and the filter method presented in the next section.
To this end, we first consider a lumped approximation of the non-mortar matrix $\nmMatN$
\begin{equation*}
  (\inmMat_N)_{pq} \colonequals \begin{cases} \sum\limits_{j=1}^{\mone} (\nmMatN)_{pj} & p=q,\\ 0 & \text{else.} \end{cases}
\end{equation*}
We then define a new transformation $\iT(\varphik)$ by formula \eqref{eq:exact_mortar_transformation}, but using the diagonal matrix $\inmMat_N^{-1}$ instead of $\nmMatN^{-1}$.
Then, we apply this new transformation to the tangent stiffness matrix $\Hk$ only, but we keep the exact transformation $\T(\varphik)$ for the gradient $\fkT$.
%
%\begin{equation}
%  \Hk \approx \iHk \colonequals (\Tk)^{-T}(\iTk)^T\Hk\iTk(\Tk)^{-1}.
%  \label{eq:inexact_hessian}
%\end{equation}
%
The resulting approximate SQP problem reads
\begin{equation}
    \tag{IQP}
    \min_{\tvu\in\Rdn}\, \imk(\tvu),\quad (\tvnmuO)_p \leq \ag_p(\varphik),\quad p=1,\ldots,\mone,
  \label{eq:isqp_subproblem}
\end{equation}
with
\begin{equation*}
    \imk(\tvu)\colonequals \fkT\tvu + \frac 12 \tvu^T \iHkT\tvu,\quad \iHkT \colonequals \iT(\varphik)^T\iHk\iT(\varphik).
\end{equation*}
In other words, we still compute the sub-problem in the transformed coordinates of \Cref{ssec:decoupling_exact_constraints}, but we have replaced the tangent matrix by a sparser approximation.
Note that we retain the first-order consistency of the SQP model \eqref{eq:sqp_subproblems},
because the linear term $\fkT$ is still transformed according to the exact mapping $\Tk$.
This guarantees the convergence of the SQP method, and of the filter--trust-region method presented in the next section.
Further, this transformation can be done without explicitly computing $\nmMatN^{-1}$ by solving the small linear system
\begin{equation*}
  (\nmMatN(\varphik))^T\bar f = f^1_{C,0},
\end{equation*}
where $f^1_{C,0}$ consists of the first components of the entries of $\nabla\J(\varphik)$ that correspond to degrees of freedom on the non-mortar contact boundary $\gamma_h^1$.
The transformed gradient $\fkT$ can then be directly computed from $\bar f$ by multiplication with $O(\varphik)$ and $(\nmMatT(\varphik))^T$ resp. $(\mMat(\varphik))^T$,
cf.~\eqref{eq:exact_mortar_transformation}.
Similarly, the transformation back to Euclidean coordinates
\begin{equation*}
 \vu = T(\varphik)\tvu ,
\end{equation*}
can be computed without the explicit inverse $\nmMatN^{-1}$ by solving the small linear system
\begin{equation*}
  \nmMatN(\varphik)\hat\vu^1_{C,0} = -(\tvnmu + \nmMatT(\varphik)\tvu_{C,T}^1 + \mMat(\varphik)\tvu_{C}^2),
\end{equation*}
and rotating the block vector $\R^{d\mone} \ni w\colonequals\begin{pmatrix}\hat\vu^1_{C,0} &\tvu_{C,T}^1\end{pmatrix}$
%with first components $\vu^1_0$ and $(d-1)$ components $\tvu_{C,T}^1$
%
\begin{equation*}
  \vnmu = O w,\quad \vmu = \tvu_{C}^2,\quad \vu^I = \tvu^I,
\end{equation*}
where $\tvu_{C,T}^1\in\R^{(d-1)\mone}$ denotes the block-vector corresponding to the tangential non-mortar degrees of freedom.
\begin{remark}
Approximating the algebraic problem is often done in large deformation contact problems to simplify the unknown contact forces that show up explicitly in the weak formulation when applying an active-set method \cite{Hartmann_Ramm:2008,Popp:2011}.
  This allows to eliminate the Lagrange multipliers at the cost of losing angular momentum conservation.
  In contrast, by conserving the first-order consistency of the sub-problems \eqref{eq:sqp_subproblems}, the approximation of the Hessian suggested here is only affecting the convergence rate of the SQP method and preserves the angular momentum.
\end{remark}
\section{Globalisation by Filter--Trust-Region Methods}
\label{sec:filter}
We globalise the SQP method of the previous chapter by extending it to a filter--trust-region method.
In contrast to the active-set strategies widely used in contact mechanics~\cite{Hesch_Betsch:2009,Popp:2011, Hartmann_Ramm:2008}, this method can be shown to converge globally even for rather general non-convex strain energy functionals.
\subsection{Filter--trust-region methods}
The SQP method of the previous chapter converges only locally.
Furthermore, away from local minimisers of $\aJ$, the exact Hessian \eqref{eq:hessian_lagrangian} does not have to be positive definite.
Hence, approximating it by a positive definite matrix may result in poor performance of the SQP method~\cite{Nocedal:2006}.
In the following we will use the popular approximation of~\eqref{eq:hessian_lagrangian}
by the Hessian of the energy
\begin{equation*}
  \Hk = \nabla^2\aJ(\varphik),
\end{equation*}
which avoids the need to compute the Lagrange multipliers during the SQP iteration.
To handle the possible unboundedness from below of the local problems \eqref{eq:sqp_subproblems}
with this definition of $\Hk$, the \emph{trust-region globalisation} adds a norm constraint on the correction
\begin{equation}
  \norm{\tvu} \leq \Deltak,\quad k=0,1,\ldots
  \label{eq:tr_constraints}
\end{equation}
We choose the infinity norm as then \eqref{eq:tr_constraints} is equivalent to a set of bound constraints, which fits naturally with the non-smooth multigrid solver of \Cref{ssec:mg_obstacle_problems}.

The constraint is adjusted dynamically according to how well the local model approximates the non-linear functional.
We measure the approximation quality by the scalar quantity
\begin{equation}
\label{eq:model_approximation_quality}
  \rho^k \colonequals  \frac{\aJ(\varphik) -\aJ(\varphik+\vuk)}{\imk(0) - \imk(\tvuk)},
\end{equation}
where $\tvuk$ is the solution of the SQP sub-problem \eqref{eq:isqp_subproblem} in transformed coordinates and $\vuk=\Tk\tvuk$.

Incorporating \eqref{eq:tr_constraints} into \eqref{eq:isqp_subproblem} yields the constrained quadratic optimisation problems
\begin{equation}
    \tag{TRQP}
  \begin{aligned}
    &\min_{\tvu\in\Rdn} \imk(\tvu),\\
    -\Deltak  &\leq \tvu_p \leq \ag_{p}^{\Deltak},\quad 1\leq p\leq dn,
  \end{aligned}
  \label{eq:tr_subproblem}
\end{equation}
with
\begin{equation*}
  \label{eq:full_tr_constraints}
  \ag_{p}^{\Deltak} \colonequals\begin{cases} \min\cbrackets{\ag_p(\varphik),\Deltak} & p\text{ first component of degree of freedom on }\gammaH^1, \\ \Deltak &\text{else.} \end{cases}
\end{equation*}
These problems always have at least one solution, even if $\imk$ is non-convex.

To arrive at a globally convergent scheme one also has to control the possible infeasibility of the intermediate iterates $\varphik$, which results from replacing the non-linear contact constraint from \eqref{eq:algebraic_problem} by a linearised one.
We measure the infeasibility of an iterate using the non-smooth function
\begin{equation*}
  \vartheta(\avarphi)\colonequals\max_{p=1,\ldots,\mone}\cbrackets{0,-\ag_p(\avarphi)}.
\end{equation*}
A filter method creates tentative new iterates by solving \eqref{eq:tr_subproblem}, and accepting or rejecting them based on a set of criteria.
In the following we use the abbreviations $\aJk\colonequals \aJ(\varphik)$ and $\vartetk\colonequals \vartheta(\varphik)$ to denote the energy and infeasibility of the \mbox{$k$-th} iterate.
Let $\varphikone$ be a potential new iterate, i.e., $z^{k+1} = z^k + u^k$, with $u^k$
an approximate solution of \eqref{eq:tr_subproblem}.
If $\aJkone\leq\aJi$ and $\vartetkone\leq\varteti$ for all previous iterates $i$, then the step can be accepted.
If there is a previous iterate $\avarphi^i, i\leq k$, such that
\begin{equation*}
  \aJi\leq \aJkone\quad \text{ and }\quad \varteti\leq\vartetkone,
\end{equation*}
then the candidate $\varphikone$ should be rejected.
The critical question is what to do if
\begin{equation*}
  \aJkone < \aJi, \quad \text{ but } \quad \vartetkone>\varteti,
\end{equation*}
or vice versa, for all previous iterates.
To overcome this difficulty \citeauthor{Fletcher_Leyffer:2002} introduced the notion of a \textit{filter}~\cite{Fletcher_Leyffer:2002}.
\begin{definition}
  \label{def:filter}
  Let $0 <\xi <1$. A pair $(\aJk,\vartetk)$ \textit{$\xi$-dominates} $(\aJi,\varteti)$ if
\begin{equation*}
  \aJk < \aJi - \xi \vartetk \quad \text{ and }\quad \vartetk < (1-\xi)\varteti .
\end{equation*}
For a fixed constant $0 < \xi < 1$, a set of tuples $(\aJi, \varteti)$ is called a \textit{filter} $\F_{\xi}$, if no tuple $\xi$-dominates any other tuple in $\F_\xi$ (\Cref{fig:filter}).
\end{definition}
A filter defines a region of acceptable new iterates.
\begin{definition}
  \label{def:accetable}
An iterate $\varphikone$ is \textit{acceptable to the filter} $\F_\xi$, if
\begin{equation*}
  \label{eq:filter_acceptance}
 \aJ(\varphikone) < \aJi - \xi \vartheta(\varphikone)\quad \text{ or }\quad \vartheta(\varphikone) < (1-\xi)\varteti \quad\forall\, (\aJi, \varteti)\in\,\F_\xi.
\end{equation*}
\end{definition}
\begin{figure}[h]
  \centering
  \def\svgwidth{0.5\textwidth}
  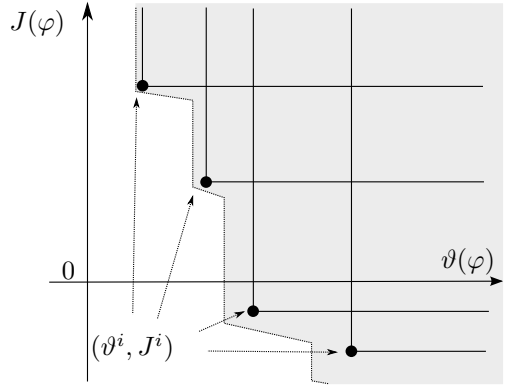
  \caption{Illustration of a filter with four points. The grey area corresponds to points that are not acceptable.}
  \label{fig:filter}
\end{figure}
Certain acceptable iterates are added to the filter during the filter iteration, and
all pairs that are dominated by the new iterate are removed.
\begin{remark}
  This criterion guarantees the convergence towards the feasible set $\K_h$ of every acceptable sequence of iterates that is subsequently added to the filter, if $\xi >0$, see~\cite[Lemma.\,15.5.2]{Conn_Gould_Toint:2000}.
\end{remark}
\begin{figure}[h]
  \centering
  \def\svgwidth{\textwidth}
  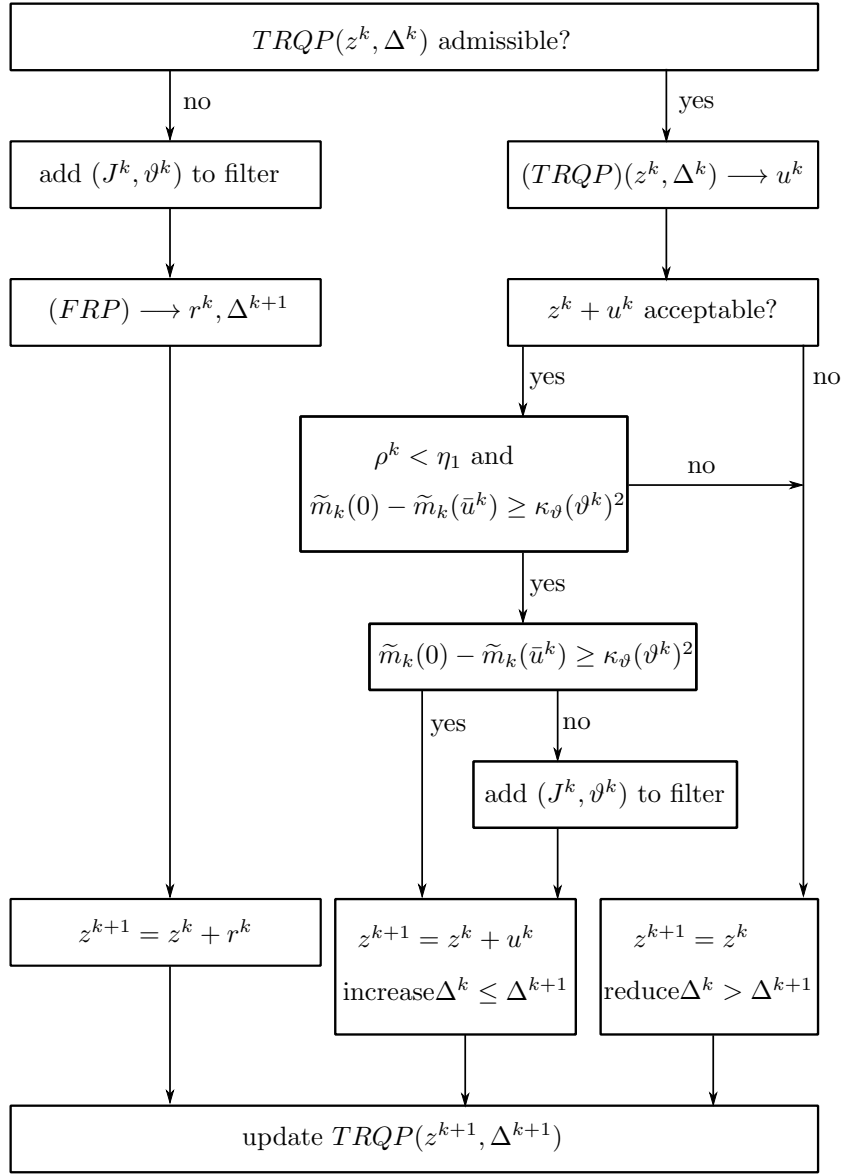
  \caption{Illustration of the filter--trust-region method}
  \label{fig:flowchart_filter}
\end{figure}
The filter--trust-region algorithm is given by the following steps:

\medskip

\noindent\textit{1.~Computing a candidate}\\
Compute a new candidate $\varphik + \vuk$ by approximately solving \eqref{eq:tr_subproblem}, and evaluate the corresponding energy $\aJ(\varphik + \vuk)$ and infeasibility $\vartheta(\varphik + \vuk)$.

\medskip

\noindent\textit{2.~Acceptance tests}\\
If the candidate is not acceptable to the filter then the trust-region is decreased.
%by some fixed factor $ 0 < \beta_1 < 1$.
Further, if the approximation quality of the model is poor, i.e., $\rho^k < \eta_1 < 1$,
for some fixed constant $0 < \eta_1 < 1$,
the candidate is also rejected whenever the current infeasibility is small.
This is estimated by checking if
\begin{equation}
  \label{eq:theta_step}
         \imk(0) - \imk(\tvuk) \geq \kteta(\vartetk)^2
\end{equation}
for some fixed $0 < \kteta < 1$. In the affirmative case the trust-region is also decreased $\Delta^{k+1} < \Delta^k$.

\medskip

\noindent\textit{3.~$\vartheta$-type iteration}\\
If the feasibility check \eqref{eq:theta_step} fails, the previous iterate $\varphik$ is added to the filter and the candidate is accepted by the filter method.
Hence \eqref{eq:theta_step} enables the method to accept candidates that improve the infeasibility
$\vartetkone <\vartetk$ while possibly increasing the energy.
This is called a $\vartheta$-type iteration.

\medskip

\noindent\textit{4.~$\aJ$-type iteration}\\
If \eqref{eq:theta_step} is fulfilled and the approximation quality of the model is high, i.e., $\rho^k \geq \eta_2$, with $0 <\eta_1 \leq \eta_2 < 1$, then additionally the trust-region
radius can be increased.
This potentially allows to achieve a larger energy reduction in the following iteration.

\medskip

\noindent\textit{5.~Ensuring admissibility}\\
The combination of the trust-region constraints with the linearised non-penetration constraints can lead to local problems \eqref{eq:tr_subproblem} that do not have a solution.
This happens when the infeasibility is too large while the trust-region is very small
\begin{equation*}
  \ag_p(\varphik) < -\Deltak\quad\text{ for some } p\in\cbrackets{1,\ldots, \mone}.
\end{equation*}
This case is treated by the filter method as follows:
First, the tuple $(\aJ^k,\vartetk)$ of the previous iterate is added to the filter;
it is always acceptable by construction.
Then, the algorithm enters the so-called \emph{feasibility restoration phase}.
In this phase a new iterate $\varphikone$ and trust-region radius $\Delta^{k+1}$
are computed such that $\varphikone$ is acceptable to the filter and the local problem \eqref{eq:tr_subproblem} is admissible again.
This is done by minimising the infeasibility directly
\begin{equation}
  \tag{FRP}
  \label{eq:frp_problem}
  \min_{\avarphi\in\Rdn}\vartheta(\avarphi),
\end{equation}
e.g., by using a semi-smooth trust-region method~\cite{Conn_Gould_Toint:2000}.
To ensure that a point which is acceptable to the filter can be computed, it is crucial that only infeasible points are included in the filter
%On the other hand, to enable the $(FRP)$ to find an acceptable new iterate, it is crucial that the filter does only contain infeasible points, i.e.,
%
\begin{equation*}
(\aJi, \varteti)\in\F\Rightarrow\varteti\neq 0.
\end{equation*}
This is achieved by only adding the tuple $(\aJ^k,\vartetk)$ to the filter if \eqref{eq:theta_step} fails.
%When this fails the filter--trust-region methods stops unsuccessfully.
%Otherwise the resulting candidate $\varphikone$ is accepted and the method continues with the next iteration.\\
A flowchart of the method can be found in \Cref{fig:flowchart_filter}.

\subsection{Global convergence of the filter--trust-region method}
\label{ssec:filter_convergence}
The general filter--trust-region theory shows global convergence of the method to first-order optimal
points under mild assumptions on the problem~\cite{Fletcher:2002b}. We state these assumptions here
for the case of the finite-strain contact problem, and then formally state the convergence result.

\begin{assumption}
  \label{ass:bounded_iterates}
  The iterates $\varphik$ generated by the filter method stay in a compact set $\mathcal{L}$.
\end{assumption}

Unfortunately, this does not immediately follow from coercivity of the hyperelastic energy functional,
as the filter--trust-region algorithm is not a monotone descent method.

\begin{assumption}
  \label{ass:inexactjac}
  The contact constraint $c : \R^{dn} \to \R^{\mone}$ and the energy are are both
  twice continuously differentiable on $\mathcal{L}$.
  \end{assumption}

The contact constraint is smooth enough if the contact boundary is, and if the occurring deformations
are not too extreme.

\begin{assumption}
  \label{ass:trans}
  The normal non-mortar matrix $\nmMatN(\avarphi)$ is regular and has a bounded inverse on $\mathcal{L}$.
\end{assumption}

This assumption again only rules out a few extreme deformations.  As $\nmMatN$ is a mass matrix,
the assumption is mainly about the grid quality of the deformed configurations.

The smoothness of $\ag$ and the boundedness of $\nmMatN^{-1}$ imply the boundedness of the exact and lumped transformations that decouple the linearised contact constraints
  \begin{equation*}
    \norm{\T(\varphik)^{-1}}\leq \kT,\quad \norm{\T(\varphik)}\leq \kT,\quad \big\lVert\iT(\varphik)\big\rVert\leq \kT,
  \end{equation*}
  with a constant $\kT>0$ independent of $k$.
This in turn implies the boundedness of the transformed Hessians $\big\lVert\iHkT\big\rVert$ and gradients $\big\lVert\fkT\big\rVert$
%\noindent From these assumptions the boundedness of $\norm{\nabla^2\aJ}$ and the transformed Hessian $\big\lVert\iHkT\big\rVert$ and gradient $\big\lVert\fkT\big\rVert$ follows:
%
\begin{align*}
  &\big\lVert\iHkT\big\rVert\leq \max_{\avarphi\in\mathcal{L}}\norm{\nabla^2 \aJ(\avarphi)}\kT^2,\quad
  \norm{\fkT} \leq \max_{\avarphi\in\mathcal{L}}\,\norm{\nabla \aJ(\avarphi)}\kT.
\end{align*}
This last boundedness is the assumption that typically appears in general filter--trust-region results.

The final assumption is to ensure that the model energy $\imk$ is reduced sufficiently during each filter iteration.
Therefore, let $\chi:\Rdn\to\R$ be an optimality measure of the sub-problem \eqref{eq:tr_subproblem}, i.e., a non-negative, continuous function that vanishes
if and only if $\tvu$ is a stationary point of the inexact SQP sub-problem \eqref{eq:tr_subproblem}.
\begin{assumption}
\label{ass:cauchy_decrease}
  The numerical solution $\tvuk$ of \eqref{eq:tr_subproblem} fulfils the \emph{sufficient Cauchy decrease} condition:
    \begin{equation*}
      \imk(0) - \imk(\tvuk) \geq \kscd\chi(\varphik) \min\cbrackets{\frac{\chi(\varphik)}{\bigl\lVert\iHkT\bigr\rVert},\,\Deltak},
    \end{equation*}
    for some constant $\kscd >0$.
\end{assumption}
  \Cref{ass:cauchy_decrease} means that at least a fixed fraction of the decrease that is generated
  by following the projected gradient has to be achieved \cite{Conn:1993}.
  This assumption is fulfilled when suitably many iterations of the globally convergent method TNNMG method are performed to solve \eqref{eq:tr_subproblem}.
Numerical tests indicate that already one iteration is enough to exceed the desired decrease.

From these assumptions the general filter--trust-region theory~\cite{Conn_Gould_Toint:2000,Fletcher:2002b}
deduces the following global convergence result.
\begin{theorem}
  [\cite{Conn_Gould_Toint:2000}, Theorem\,15.5.13]
  \label{thm:filter_global_convergence}
    Let \Cref{ass:bounded_iterates,ass:trans,ass:inexactjac,ass:cauchy_decrease} hold and $(\varphik)_{k\in\Na}$ be a sequence generated by the filter--trust-region method.
    Then, either the feasibility restoration phase terminates unsuccessfully by converging to a critical point of \eqref{eq:frp_problem} or there exists a subsequence $(\avarphi^{k_l})_{l\in\Na} \subseteq (\varphik)_{k\in\Na}$ such that
    \begin{equation*}
      \lim_{l\to\infty} \avarphi^{k_l} = \avarphi_*,
    \end{equation*}
    where $\avarphi_*$ is a first-order critical point of the non-linear problem \eqref{eq:algebraic_problem}.
\end{theorem}
As all four assumptions are reasonable in the context of finite-strain contact problems, we have shown
global convergence of our filter--trust-region multigrid contact solver.

\subsection{Multigrid solution of the trust-region sub-problems}
\label{ssec:tnnmmg}
The TNNMG method  of \Cref{ssec:mg_obstacle_problems} cannot be used to solve the trust-region sub-problems \eqref{eq:tr_subproblem}.
They are still quadratic minimisation problems with bound constraints; however,
the functionals may now be non-convex. Remember that TNNMG includes an unconstrained minimisation step for the truncated energy \eqref{eq:convex_defect_problem}.
This minimisation does not have a solution if the quadratic energy is not convex.
We circumvent this problem by adding an additional set of bound constraints to \eqref{eq:convex_defect_problem},
which can be interpreted as applying a trust-region method to compute the coarse grid correction.
The resulting obstacle problem can be solved using the classical monotone multigrid method (MMG) from \cite[Algorithm 5.10]{Graeser_Kornhuber:2009}.
In contrast to the TNNMG, the monotone multigrid method does not neglect the obstacles on the coarser grids.
In principle, MMG could be used directly to solve \eqref{eq:tr_subproblem}.
However, this would require the transformations \eqref{eq:exact_mortar_transformation} on each level of the grid hierarchy, which complicates the implementation~\cite{Wohlmuth_Krause:2003,Sander:2008}.

In the following we revisit the sub-steps of a TNNMG iteration and describe the necessary modifications.
We use $\nu$ to denote the TNNMG iteration number, but for brevity we omit the SQP iteration index $k$.
Let $\vuknu\in\Rdn$ be the current iterate.

\medskip

\noindent\textit{1.~Projected Gauss-Seidel step}\\
The non-linear smoother remains unchanged, noting that the one-dimensional minimisation problems \eqref{eq:gs_problem_convex} always have at least one solution, because we minimise over a compact set now.
If the global minimiser \eqref{eq:gs_problem_convex} is not unique, we pick the one with a larger $\alpha_p$.
Let $\tvu^{\nu+\frac 12}$ denote the resulting pre-smoothed iterate.

\medskip

\noindent\textit{2.~Truncated linear correction}\\
We then set up the truncated defect problem \eqref{eq:convex_defect_problem}. In transformed coordinates it reads
\begin{equation*}
  v^\nu \colonequals \argmin_{\tcorr\in\Rdn}\,\frac 12\tcorr^TQ^\nu \widetilde{H}_T Q^\nu\tcorr - (r_T^\nu Q^\nu)^T \tcorr,
\end{equation*}
where
\begin{equation*}
  r^\nu_T\colonequals f_T-\widetilde{H}_T\tvu^{\nu +\frac 12},
\end{equation*}
and $Q^\nu \colonequals Q(\tvu^{\nu+\frac 12})$ is the truncation matrix \eqref{eq:truncation_matrix}.
Next, one transforms the defect problem back into Euclidean coordinates
\begin{equation}
  \min_{\corr\in\Rdn}\,\frac 12\corr^T \widetilde{H}^\nu\corr - (r^\nu)^T \corr,
  \label{eq:truncated_defect_problem}
\end{equation}
with
\begin{equation*}
  r^\nu\colonequals r^\nu_TQ^\nu\T^{-1},\quad \widetilde{H}^\nu\colonequals\T^{-T}Q^\nu\widetilde{H}_T^\nu Q^\nu\T^{-1}.
\end{equation*}
to avoid multigrid prolongation operators in transformed coordinates.
To handle the possible unboundedness of the defect problem \eqref{eq:truncated_defect_problem} we additionally prescribe a set of finite bound constraints
\begin{equation}
  \label{eq:weighted_constraints}
  a_i  \leq \corr_i \leq b_i \quad i=1,\ldots,dn,
\end{equation}
which lead to a minimisation problem on a compact set.

The constraints are constructed such that a correction $v^\nu$ in untransformed
coordinates that complies with~\eqref{eq:weighted_constraints} will not violate
the trust-region constraints of~\eqref{eq:tr_subproblem} when converted to
transformed coordinates.
\begin{lemma}
Let $w_i$ denote the number of non-zero entries in the $i$-th row of the sparse transformation matrix $T^{-1}$ and let
\begin{equation*}
  R_j\colonequals\bigl\{ 1 \le i \le dn \; : \; T^{-1}_{ij}\neq 0\bigr\}
\end{equation*}
for all $l = 1,\dots, dn$.
Let
\begin{equation}
    a_j \colonequals\max_{i\in R_j}\biggl\{\frac{-\operatorname{sign}(\T^{-1}_{ij})\Deltak-\tvu_i^{\nu+\frac 12}}{w_i\T^{-1}_{ij}}\biggr\},\quad
    b_j \colonequals\min_{i\in R_j}\biggl\{\frac{\operatorname{sign}(\T^{-1}_{ij})\Deltak-\tvu_i^{\nu+\frac 12}}{w_i\T^{-1}_{ij}}\biggr\}.
  \label{eq:max_weights}
\end{equation}
Then if $\corr\in\Rdn$ is such that \eqref{eq:weighted_constraints} holds, we get
\begin{equation}
\label{eq:shifted_tr_constraint}
  \lVert\tvu^{\nu+\frac 12}+\tcorr\rVert_{\infty}\leq\Deltak.
\end{equation}
\end{lemma}
\begin{proof}
  Insert \eqref{eq:max_weights} into \eqref{eq:weighted_constraints} to obtain
\begin{equation*}
\max_{i\in R_j}\cbrackets{\frac{-\operatorname{sign}(\T^{-1}_{ij})\Delta_k-\tvu^{\nu+\frac 12}}{w_i\T^{-1}_{ij}}} \leq v_j \leq \min_{i\in R_j}\cbrackets{\frac{\operatorname{sign}(\T^{-1}_{ij})\Delta_k-\tvu^{\nu+\frac 12}}{w_i\T^{-1}_{ij}}}.
\end{equation*}
Now consider the $p$-th constraint in decoupling coordinates
\begin{align*}
  \tvu^{\nu+\frac 12}_p + \tcorr_p &= \tvu^{\nu+\frac 12}_p + \sum_{j = 1}^{dn} \T_{pj}^{-1}\corr_j\\
  &\leq \tvu^{\nu+\frac 12}_p +\sum_{\substack{j=1\\\T^{-1}_{pj}\neq 0}}^{dn}\T_{pj}^{-1}\frac{\Delta_k-\tvu^{\nu+\frac 12}_p}{w_p\T_{pj}^{-1}} = \Delta_k,
\end{align*}
which is the upper bound of \eqref{eq:shifted_tr_constraint}.
The lower bound is shown in the same way.
\end{proof}
We construct the defect problem constraints by replacing the feasible weights \eqref{eq:max_weights} by  an averaged version
\begin{equation}
  \begin{aligned}
  &a_j \colonequals\frac 1{\abs{R_j}}\sum_{i\in R_j}\biggl\{\frac{-\operatorname{sign}(\T^{-1}_{ij})\Deltak-\tvu_i^{\nu+\frac 12}}{w_i\T^{-1}_{ij}}\biggr\},\\
  &b_j \colonequals\frac 1{\abs{R_j}}\sum_{i\in R_j}\biggl\{\frac{\operatorname{sign}(\T^{-1}_{ij})\Deltak-\tvu_i^{\nu+\frac 12}}{w_i\T^{-1}_{ij}}\biggr\},
\end{aligned}
  \label{eq:constructed_weights}
\end{equation}
which is less restrictive than \eqref{eq:max_weights} while capturing the scaling of the decoupling transformation $\T$.

For the approximate solution of the defect problem \eqref{eq:truncated_defect_problem}  with constraints \eqref{eq:weighted_constraints} and \eqref{eq:constructed_weights},
a standard monotone multigrid method is applied~\cite{Graeser_Kornhuber:2009}.
Like in the pre-smoothing step, the Gauss--Seidel smoothers of that method have to take into account the possible non-convexity of the local one-dimensional problems.

\medskip

\noindent\textit{3.~Projection}\\
The resulting correction $\corr^\nu$ is transformed back into the coordinates in which the linearised non-penetration constraints decouple
\begin{equation*}
  \tcorr^\nu \colonequals \T^{-1}\corr^\nu.
\end{equation*}
This transformed correction is
then projected onto the defect obstacles of \eqref{eq:tr_subproblem}, i.e., we define $\hat\corr^\nu$ by
\begin{equation*}
  \hat\corr^\nu_p\colonequals
  \begin{cases}
    %-\Deltak-\tvu_p^{\nu+\frac 12} & \text{if}\quad \tcorr_p^\nu<-\Deltak -\tvu_p^{\nu+\frac 12},\\
    \ag^{\Deltak}_p-\tvu_p^{\nu+\frac 12} & \text{if}\quad \tcorr_p^\nu>\ag^{\Deltak}_p -\tvu_p^{\nu+\frac 12},\\
    \hat\corr^\nu_p & \text{else}.
  \end{cases}
%\label{eq:feasible_projection}
\end{equation*}
\medskip

\noindent\textit{4.~Line search}\\
The tentative new iterate $\tvu^{\nu+\frac 12} + \hat\corr^\nu$ is feasible, but it may violate the monotonicity of the TNNMG method.
We ensure energy decrease by performing an exact line search in the direction of $\hat\corr^\nu$, as in \eqref{eq:line_search_problem}.
This is a scalar, quadratic, possibly non-convex minimisation problem.
Since it is posed on a closed interval it is guaranteed to have a solution, which can be computed explicitly.

The modified TNNMG algorithm converges globally towards first-order optimal points of the constrained quadratic minimisation problem \eqref{eq:tr_subproblem}.

A proof for the case without truncation is given in~\cite{Youett:2016}.

\begin{theorem}
  The TNNMG method with a monotone multigrid correction described in \Cref{ssec:tnnmmg} either stops at a first-order optimal point of \eqref{eq:tr_subproblem} or the limit of every convergent subsequence is first-order optimal.
\end{theorem}

\section{Numerical Examples}
\label{sec:numerics}
In this section we illustrate the robustness and global convergence of the filter--trust-region method.
The numerical simulations were done using the \textsc{Dune} framework~\cite{duneII:2008,Blatt:2016}.
A detailed description of the implementation of the linearised non-penetration constraint \eqref{eq:linearised_constraints} can be found in \cite{Popp:2011, Puso_Laursen:2004}.

\subsection{Ironing}
\label{ssec:die_example}

The ironing problem is often used to test the robustness of the mortar discretisation and the applied algebraic solver~\cite{Puso_Laursen:2004}.
In this example a rectangular block is placed under a half-pipe (\Cref{fig:init_ironing}).

\begin{figure}
  \centering
  \def\svgwidth{0.35\textwidth}
  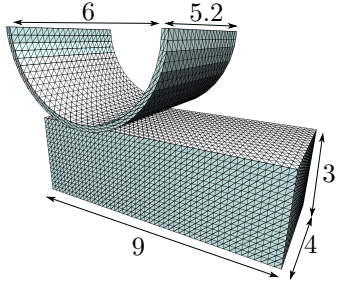
  \caption{The initial configuration of the refined grids}
  \label{fig:init_ironing}
\end{figure}
The block is fixed at the bottom with homogeneous Dirichlet conditions. For the half-pipe, non-homogeneous Dirichlet conditions are prescribed on the top boundary:
First, the half-pipe is pressed vertically into the block with a prescribed total displacement of $1.4$ units (Phase~1).
Then, in a second phase, it is swiped over the block horizontally for $2.1$ units, see \Cref{fig:ironing_evolution}.
\begin{figure}[ht]
  \centering
  \includegraphics[height=3.5cm]{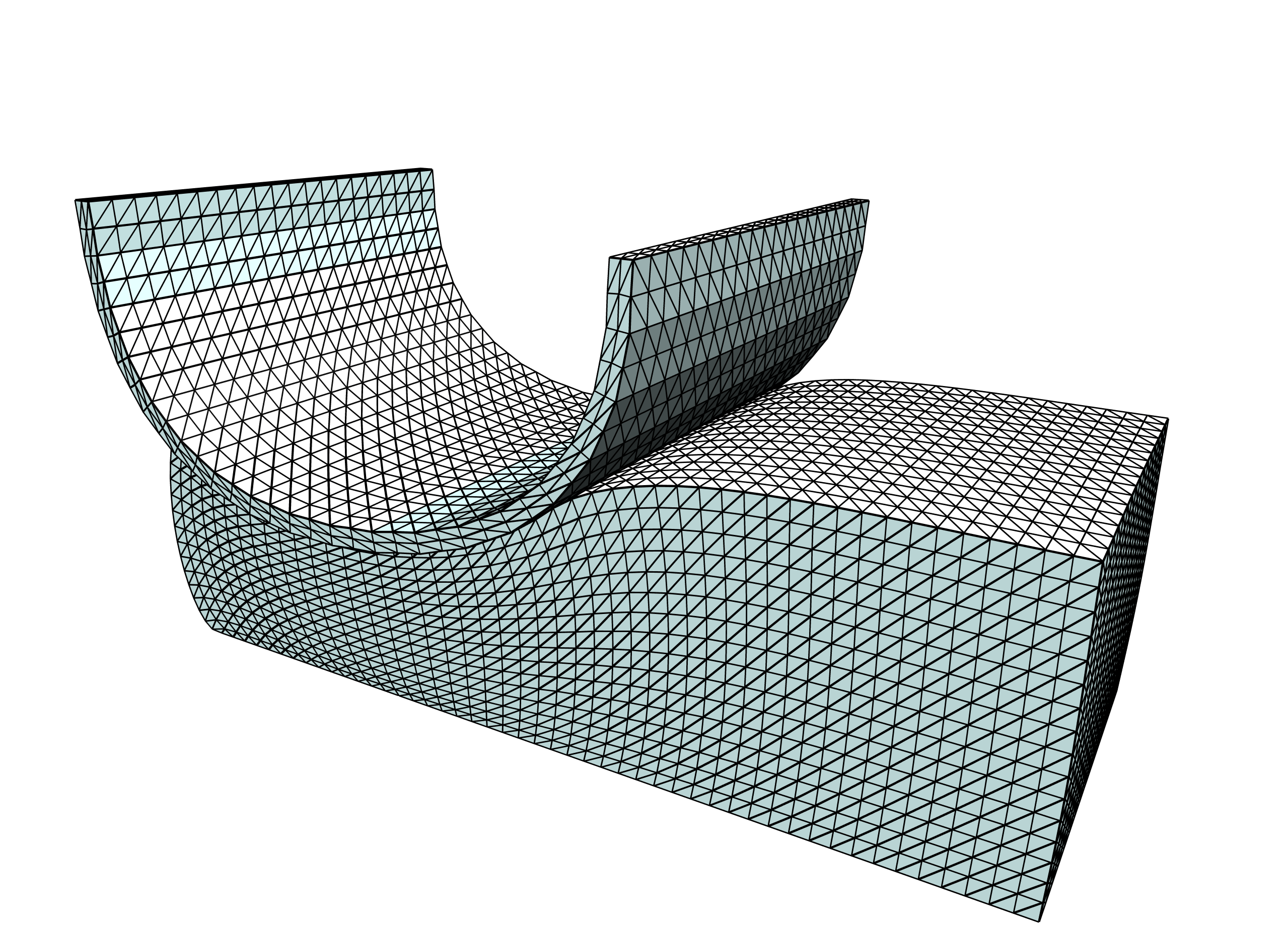}
  \includegraphics[height=3.5cm]{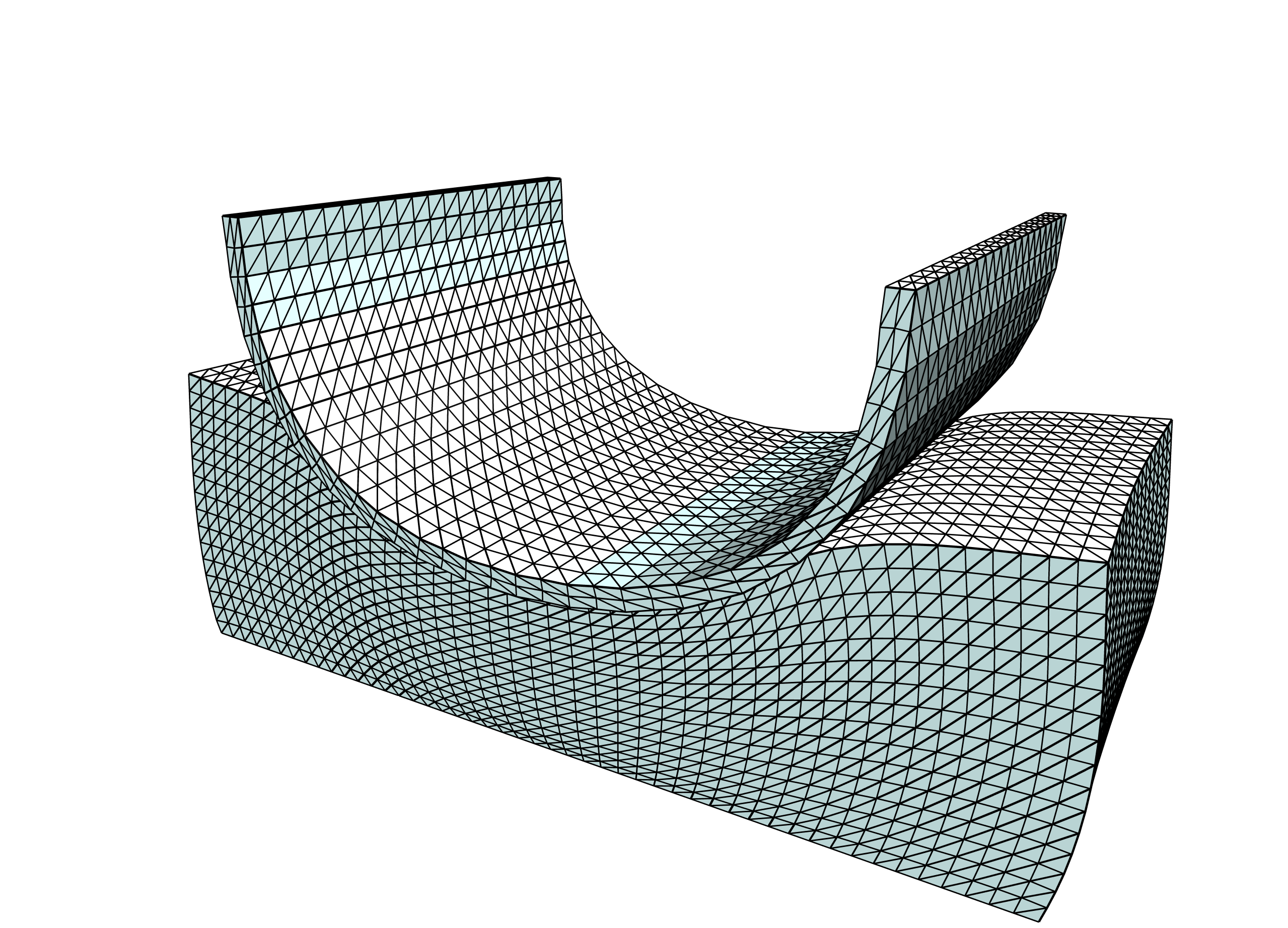}
  \caption{Left: Deformed grids after the vertical displacement. Right: Deformed grids after the horizontal displacement}
  \label{fig:ironing_evolution}
\end{figure}

This benchmark problem is usually solved in small loading steps to stabilise the widely used active-set and penalty methods,
which only converge locally~\cite{Hartmann_Ramm:2008,Hesch_Betsch:2009,Popp:2011,Puso_Laursen:2004}.
To depict the superior robustness of the proposed method, we solve it in only two steps, one for each of the two phases.
In both cases we choose the block to be the non-mortar body.
The bodies are modelled by the non-linear homogeneous Neo-Hookean material law
\begin{equation*}
  \mathcal{W}(\nabla \vvarphi) = \frac \lambda 4\bigl(\deter(\nabla\vvarphi)^2-1\bigr) -\Bigl(\frac \lambda 2 +\mu\Bigr)\log\bigl(\deter(\nabla\vvarphi)\bigr) + \mu\operatorname{tr} E(\nabla \vvarphi),
  \label{eq:neohookean_material}
\end{equation*}
where $E(\nabla \vvarphi) \colonequals \frac 12 (\nabla\vvarphi^T\nabla\vvarphi -\Id)$ denotes the Green--Lagrange strain tensor, and we choose the Lam\'e parameters as
\begin{alignat*}{3}
  \lambda_{\text{pipe}} & = 450,&\qquad\mu_{\text{pipe}} & =225,\\
  \lambda_{\text{block}} & = \frac 34,&\mu_{\text{block}} & =\frac 38.
\end{alignat*}
We use tetrahedral grids with $42\,483$ and $6\,993$ degrees of freedom, respectively,
obtained by four and one step, respectively, of uniform refinement, of corresponding
coarser grids.

The two problems are solved by the filter--trust-region method until
the $H^1$-norm of the relative correction is less than $10^{-7}$.
For the solution of the sub-problems \eqref{eq:tr_subproblem}, we apply the extended TNNMG method from \Cref{ssec:tnnmmg} until the $H^1$-norm of the relative correction falls below a tolerance of $10^{-4}$,
and we use the \textsc{IpOpt} interior point algorithm~\cite{Ipopt} to solve the problem on the coarsest grid level.

In the filter--trust-region method we used the following constants suggested in \cite{Conn_Gould_Toint:2000}: To measure the approximation quality we set $\eta_1 = 0.1$ and $\eta_2 = 0.9$, and in the $\vartheta$-type criterion \eqref{eq:theta_step} we use $\kteta=10^{-4}$.
When the trust region radius needs to be decreased we use
\begin{equation*}
  \Delta^{k+1} = 0.25\min\cbrackets{\norm{\tvuk}_\infty,\,\Deltak},
\end{equation*}
and we skip increasing it during $\aJ$-type iterations $\Delta^{k+1} = \Delta^k$.
As initial trust-region we chose $\Delta^0 = 0.5$ for both phases.
To monitor the convergence of the method towards first-order optimal points, we consider the optimality measure
\begin{equation}
  \chi(\varphik) \colonequals  \biggl|\min_{\substack{\bar{d}^1_{C,0} \leq \ag(\varphik)\\
                        \norm{\bar d}_\infty\leq 1}} \scp{\nabla\imk(0),\bar d}\biggr|,
  \label{eq:optimality_measure}
\end{equation}
which vanishes for first-order optimal points of $\aJ$ when additionally $\vartheta(\varphik)\to 0$
(see \cite[Th.\,12.1.6 \& Thm.\,15.5.13]{Conn_Gould_Toint:2000}).
Its evaluation involves a linear minimisation problem with bound constraints, which can be solved easily.
\begin{figure}[h]
  \centering
  \input{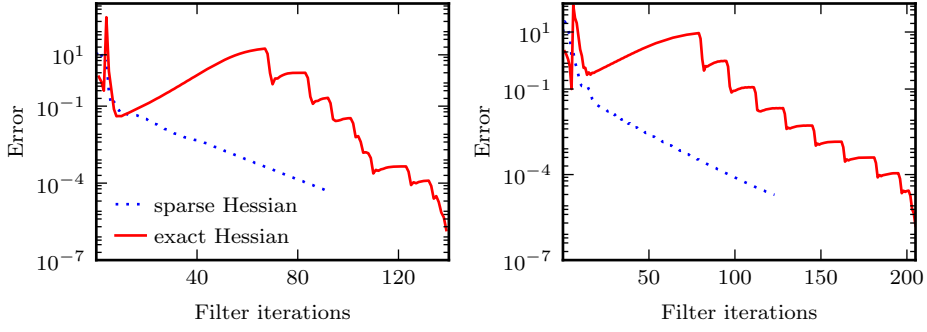}
  \caption{
    Convergence of the optimality measure $\chi$ for the filter--trust-region method with inexact sparse Hessians
    compared to the filter method with exact Hessians.
    Left: Vertical phase. Right: Horizontal phase.}
  \label{fig:ironing_filter_convergence}
\end{figure}

\Cref{fig:ironing_filter_convergence} shows a comparison of convergence of the filter--trust-region method with sparse inexact Hessians \eqref{eq:isqp_subproblem} and the
respective method using exact Hessians \eqref{eq:sqp_transformed_subproblem}.
For the latter we constructed the decoupling coordinate transformation \eqref{eq:exact_mortar_transformation} by computing a LU-decomposition of $\nmMatN$ using UMFPack,
which leads to dense blocks in the exactly transformed stiffness matrices $\HkT$.
For both problems the total iteration numbers are comparable, but due to the sparsity of the inexact Hessian $\iHkT$, the total wall time required by the inexact version is over $80\%$ smaller than for the filter method with exact Hessians, see \Cref{tab:ironing_comp}.
For the exact (sparse) Hessian, 19 (14) iterates were rejected by the filter, and 1 (15) because of insufficient
model approximation~\eqref{eq:model_approximation_quality}.
In total 10 steps in each phase had to be recomputed for the exact Hessian method and 13 respectively 16 steps were repeated in the inexact Hessian case.
The feasibility restoration phase of the filter method never occurred.

\begin{table}[ht]
\begin{center}
\begin{tabular}{l|c|c}
   & Averaged wall time $\HkT$ & Averaged wall time $\iHkT$\\
	\hline
  Assembly of \eqref{eq:tr_subproblem} & 41.31s & 9.8s\\
	\hline
  TNNMG solution  & 73.6s & 15.5s \\
	\hline
%  Filter iteration  & 123.94s & 31.37s\\
	\hline
  Total time to solution & 17\,227s & 2\,917s
\end{tabular}
\end{center}
\caption{Averaged CPU wall times for the exact and inexact filter for the vertical phase.}
\label{tab:ironing_comp}
\end{table}

In \Cref{fig:tr_plot} the trust-region radius $\Deltak$ and the infeasibility of both variants are shown during the vertical phase.
Once the approximation quality of the sub-problems becomes too bad, the step is rejected and the trust-region is decreased to achieve a better approximation of the non-linear energy $\aJ$.
Surprisingly, in the case of the inexact Hessians, the growing instability is detected much earlier than in the case of exact Hessians, leading to a faster convergence in this test problem.
\begin{figure}[h]
  \centering
  \input{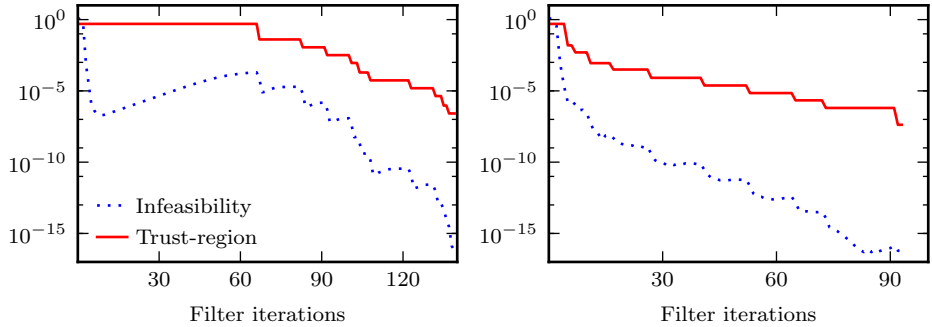}
  \caption{
  The trust-region and infeasibility of the filter method.
  Left: Vertical phase using the exact dense Hessians. Right: Vertical phase using the sparse approximation.}
  \label{fig:tr_plot}
\end{figure}
\begin{figure}[ht]
  \centering
  \input{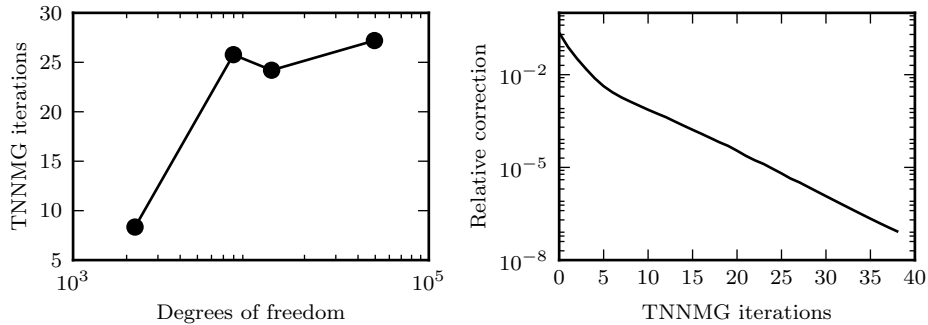}
  \caption{
    Left: TNNMG iterations averaged over the filter--trust-region iteration for an increasing number of degrees of freedom. Right: Convergence of the TNNMG method for a single problem \eqref{eq:tr_subproblem}}
  \label{fig:mg_it_plot}
\end{figure}

In the left of \Cref{fig:mg_it_plot} the average number of TNNMG iterations needed to solve the local problems \eqref{eq:tr_subproblem}, is shown for different numbers of refinement steps.
The iteration numbers appear to be bounded, which indicates the mesh independent convergence often observed for multigrid methods~\cite{Graeser_Kornhuber:2009}.
In the right of \Cref{fig:mg_it_plot} the fast convergence of the TNNMG method applied to a
single \eqref{eq:tr_subproblem} is plotted for an error tolerance of $10^{-7}$.

\section{Conclusion}

In this paper we presented a globally convergent solver for large deformation contact problems.
The solver is using a decoupling of the linearised contact constraints that allows to apply a fast and efficient multigrid method for the solution of the quadratic constrained sub-problems.
%almost mesh-independent multigrid efficiency for sub-problems
The method stands out due to its superior robustness over locally convergent methods, enabling to solve problems without applying incremental loading steps.
To improve the convergence speed, second-order consistent SQP models could be used to achieve locally super-linear convergence \cite{Ulbrich:2004}, which is part of future work.
%(\cite{Ulbrich:2004}), which is part of future work.
%Locally super-linear convergence of the proposed method can be achieved by using
%
\bibliographystyle{abbrvnat}
\bibliography{youett-sander-kornhuber-filter-contact-2017}

\begin{thebibliography}{26}
\providecommand{\natexlab}[1]{#1}
\providecommand{\url}[1]{\texttt{#1}}
\expandafter\ifx\csname urlstyle\endcsname\relax
  \providecommand{\doi}[1]{doi: #1}\else
  \providecommand{\doi}{doi: \begingroup \urlstyle{rm}\Url}\fi

\bibitem[Bastian et~al.(2008)Bastian, Blatt, Dedner, Engwer, Kl\"ofkorn,
  Kornhuber, Ohlberger, and Sander]{duneII:2008}
P.~Bastian, M.~Blatt, A.~Dedner, C.~Engwer, R.~Kl\"ofkorn, R.~Kornhuber,
  M.~Ohlberger, and O.~Sander.
\newblock {A Generic Grid Interface for Parallel and Adaptive Scientific
  Computing. Part {II}: Implementation and Tests in {DUNE}}.
\newblock \emph{Computing}, 82\penalty0 (2--3):\penalty0 121--138, 2008.
\newblock \doi{10.1007/s00607-008-0004-9}.

\bibitem[Blatt et~al.(2016)Blatt, Burchardt, Dedner, Engwer, Fahlke, Flemisch,
  Gersbacher, Gr{\"a}ser, Gruber, Gr{\"u}ninger, and Sander]{Blatt:2016}
M.~Blatt, A.~Burchardt, A.~Dedner, C.~Engwer, J.~Fahlke, B.~Flemisch,
  C.~Gersbacher, C.~Gr{\"a}ser, F.~Gruber, C.~Gr{\"u}ninger, and O.~Sander.
\newblock The distributed and unified numerics environment, version 2.4.
\newblock \emph{Archive of Numerical Software}, 4\penalty0 (100):\penalty0
  13--29, 2016.

\bibitem[Ciarlet(1988)]{Ciarlet:1988}
P.~Ciarlet.
\newblock \emph{Mathematical Elasticity, Volume~I: Three-Dimensional
  Elasticity}.
\newblock Elsevier, 1988.

\bibitem[Conn et~al.(1993)Conn, Gould, Sartenaer, and Toint]{Conn:1993}
A.~R. Conn, N.~Gould, A.~Sartenaer, and P.~L. Toint.
\newblock Global convergence of a class of trust region algorithms for
  optimization using inexact projections on convex constraints.
\newblock \emph{SIAM Journal on Optimization}, 3\penalty0 (1):\penalty0
  164--221, 1993.

\bibitem[Conn et~al.(2000)Conn, Gould, and Toint]{Conn_Gould_Toint:2000}
A.~R. Conn, N.~I.~M. Gould, and P.~L. Toint.
\newblock \emph{Trust-Region Methods}.
\newblock MPS-SIAM Series on Optimization. SIAM, 2000.
\newblock \doi{10.1137/1.9780898719857}.

\bibitem[Fletcher and Leyffer(2002)]{Fletcher_Leyffer:2002}
R.~Fletcher and S.~Leyffer.
\newblock Nonlinear programming without a penalty function.
\newblock \emph{Mathematical {P}rogramming}, 91:\penalty0 239--269, 2002.
\newblock \doi{10.1007/s101070100244}.

\bibitem[Fletcher et~al.(2002)Fletcher, Gould, Leyffer, Toint, and
  W\"achter]{Fletcher:2002b}
R.~Fletcher, N.~Gould, S.~Leyffer, P.~Toint, and A.~W\"achter.
\newblock Global convergence of a trust-region {SQP}-filter algorithm for
  general nonlinear programming.
\newblock \emph{SIAM Journal on Optimization}, 13\penalty0 (3):\penalty0
  635--659, 2002.
\newblock \doi{10.1137/S1052623499357258}.

\bibitem[Gr{\"a}ser and Kornhuber(2009)]{Graeser_Kornhuber:2009}
C.~Gr{\"a}ser and R.~Kornhuber.
\newblock Multigrid methods for obstacle problems.
\newblock \emph{Journal of Computational Mathematics}, 27\penalty0
  (1):\penalty0 1--44, 2009.

\bibitem[Gr\"aser et~al.(2009)Gr\"aser, Sack, and
  Sander]{Graeser_Sack_Sander:2009}
C.~Gr\"aser, U.~Sack, and O.~Sander.
\newblock Truncated nonsmooth {N}ewton multigrid methods for convex
  minimization problems.
\newblock In \emph{Domain Decomposition Methods in Science and Engineering
  XVIII}, volume~70 of \emph{Lecture Notes in Computational Science and
  Engineering}, pages 129--136. Springer, 2009.
\newblock \doi{10.1007/978-3-642-02677-5}.

\bibitem[Hartmann and Ramm(2008)]{Hartmann_Ramm:2008}
S.~Hartmann and E.~Ramm.
\newblock A mortar based contact formulation for non-linear dynamics using dual
  {L}agrange multipliers.
\newblock \emph{Finite Elements in Analysis and Design}, 44:\penalty0 245--258,
  2008.
\newblock \doi{10.1016/j.finel.2007.11.018}.

\bibitem[Hesch and Betsch(2009)]{Hesch_Betsch:2009}
C.~Hesch and P.~Betsch.
\newblock A mortar method for energy--momentum conserving schemes in
  frictionless dynamic contact problems.
\newblock \emph{International Journal for Numerical Methods in Engineering},
  77:\penalty0 1468--1500, 2009.
\newblock \doi{10.1002/nme.2466}.

\bibitem[H{\"u}eber and Wohlmuth(2005)]{Wohlmuth:2005}
S.~H{\"u}eber and B.~Wohlmuth.
\newblock A primal--dual active set strategy for non-linear multibody contact
  problems.
\newblock \emph{Computer Methods in Applied Mechanics and Engineering},
  194\penalty0 (27):\penalty0 3147--3166, 2005.

\bibitem[Krause and Wohlmuth(2003)]{Wohlmuth_Krause:2003}
R.~Krause and B.~Wohlmuth.
\newblock Monotone methods on nonmatching grids for nonlinear contact problems.
\newblock \emph{SIAM Journal on Scientific Computing}, 25\penalty0
  (1):\penalty0 324--347, 2003.
\newblock \doi{10.1137/S1064827502405318}.

\bibitem[Laursen(2003)]{Laursen:2003}
T.~A. Laursen.
\newblock \emph{Computational Contact and Impact Mechanics}.
\newblock Springer, 2003.
\newblock \doi{10.1007/978-3-211-77298-0}.

\bibitem[Laursen and Simo(1993)]{Laursen_Simo:1993}
T.~A. Laursen and J.~C. Simo.
\newblock A continuum-based finite element formulation for the implicit
  solution of multibody, large deformation frictional contact problems.
\newblock \emph{International Journal of Numerical Methods in Engineering},
  36\penalty0 (20):\penalty0 3451--3485, 1993.
\newblock \doi{10.1002/nme.1620362005}.

\bibitem[Nocedal and Wright(2006)]{Nocedal:2006}
J.~Nocedal and S.~Wright.
\newblock \emph{Numerical optimization}.
\newblock Springer, 2006.
\newblock \doi{10.1007/978-0-387-40065-5}.

\bibitem[Popp et~al.(2010)Popp, Gee, and Wall]{Popp:2011}
A.~Popp, M.~W. Gee, and W.~Wall.
\newblock Finite deformation contact based on 3d dual mortar and semi--smooth
  {N}ewton approach.
\newblock In \emph{Trends in Computational Contact Mechanics}, pages 57--77.
  Springer, 2010.
\newblock ISBN 978-3-642-22167-5.

\bibitem[Puso and Laursen(2004)]{Puso_Laursen:2004}
M.~Puso and T.~Laursen.
\newblock A mortar segment-to-segment contact method for large deformation
  solid mechanics.
\newblock \emph{Computional Methods in Applied Mechanics and Engineering},
  193:\penalty0 601--629, 2004.
\newblock \doi{10.1016/j.cma.2003.10.010}.

\bibitem[Sander(2008)]{Sander:2008}
O.~Sander.
\newblock \emph{Multidimensional Coupling in a Human Knee Model}.
\newblock PhD thesis, Freie Universit{{\"a}t} Berlin, 2008.

\bibitem[Tur et~al.(2009)Tur, Fuenmayor, and Wriggers]{Wriggers:2009}
M.~Tur, F.~J. Fuenmayor, and P.~Wriggers.
\newblock A mortar-based frictional contact formulation for large deformations
  using {L}agrange multipliers.
\newblock \emph{Computer Methods in Applied Mechanics and Engineering},
  198\penalty0 (37):\penalty0 2860--2873, 2009.

\bibitem[Ulbrich(2004)]{Ulbrich:2004}
S.~Ulbrich.
\newblock On the superlinear local convergence of a filter-{SQP} method.
\newblock \emph{Mathematical Programming}, 100\penalty0 (1):\penalty0 217--245,
  2004.

\bibitem[W{\"a}chter and Biegler(2006)]{Ipopt}
A.~W{\"a}chter and L.~T. Biegler.
\newblock On the implementation of an interior-point filter line-search
  algorithm for large-scale nonlinear programming.
\newblock \emph{Mathematical Programming}, 106\penalty0 (1):\penalty0 25--57,
  2006.

\bibitem[Wohlmuth(2001)]{Wohlmuth:2001}
B.~Wohlmuth.
\newblock \emph{Discretization Methods and Iterative Solvers based on Domain
  Decomposition}.
\newblock LNCSE vol. 17. Springer Verlag, 2001.
\newblock \doi{10.1007/978-3-642-56767-4}.

\bibitem[Wohlmuth(2011)]{Wohlmuth:2011}
B.~Wohlmuth.
\newblock Variationally consistent discretization schemes and numerical
  algorithms for contact problems.
\newblock \emph{Acta Numerica}, 20:\penalty0 569--734, 2011.
\newblock \doi{10.1017/S0962492911000079}.

\bibitem[Wriggers(2006)]{Wriggers:2006}
P.~Wriggers.
\newblock \emph{Computational contact mechanics}.
\newblock Springer, 2006.

\bibitem[Youett(2016)]{Youett:2016}
J.~Youett.
\newblock \emph{Dynamic large deformation contact problems and applications in
  virtual medicine}.
\newblock PhD thesis, Freie Universit{\"a}t Berlin, 2016.
\newblock URL
  \url{http://www.diss.fu-berlin.de/diss/receive/FUDISS_thesis_000000102281}.

\end{thebibliography}
\end{document}